\documentclass[a4paper,11pt]{amsart}

\usepackage{ifpdf}
\usepackage{amsmath, amsthm, amsfonts, amssymb, amscd}
\usepackage[active]{srcltx}
\usepackage{color}
\usepackage[ansinew]{inputenc}

\newtheorem{theorem}{Theorem}[section] 
\newtheorem{corollary}[theorem]{Corollary}
\newtheorem{lemma}[theorem]{Lemma}

\newtheorem{proposition}[theorem]{Proposition}
\newtheorem*{proposition*}{Proposition}

\newtheorem*{question*}{Question}
\newtheorem*{theorem*}{Theorem}
\newtheorem*{claim*}{Claim}

\newtheorem*{corollary*}{Corollary}
\newtheorem{maintheorem}{Theorem}

\theoremstyle{definition}
\newtheorem{definition}[theorem]{Definition}

\theoremstyle{remark}
\newtheorem{remark}[theorem]{Remark}

\newtheorem*{remark*}{Remark}

\newcommand{\R}{\mathbb{R}}\newcommand{\N}{\mathbb{N}}
\newcommand{\Z}{\mathbb{Z}}\newcommand{\Q}{\mathbb{Q}}
\newcommand{\T}{\mathbb{T}}\newcommand{\C}{\mathbb{C}}
\newcommand{\D}{\mathbb{D}}\newcommand{\A}{\mathbb{A}}
\renewcommand{\S}{\mathbb{S}}

  \def\cG{\mathcal{G}}  
    
   \def\cO{\mathcal{O}} 
    
\def\cE{\mathcal{E}}    
   \def\cR{\mathcal{R}}

\DeclareMathOperator{\Diff}{Diff}

\begin{document}

\author[P. Le Calvez]{Patrice Le Calvez}
\address{Institut de Math\'ematiques de Jussieu-Paris Rive Gauche, IMJ-PRG, Sorbonne Universit\'e, Universit\'e Paris-Diderot, CNRS, F-75005, Paris, France \enskip \& \enskip Institut Universitaire de France}
\curraddr{}
\email{patrice.le-calvez@imj-prg.fr}

\author[M. Sambarino]{Mart\'{\i}n Sambarino}
\address{CMAT, Facultad de Ciencias, Universidad de la Rep\'ublica, Uruguay.}
\curraddr{Igua 4225 esq. Mataojo. Montevideo, Uruguay.}
\email{samba@cmat.edu.uy}

\title[Homoclinic orbits for area preserving diffeomorphisms]{Homoclinic orbits for area preserving diffeomorphisms of surfaces}
\thanks{M.S was partially supported by CNRS-France and CSIC group 618-Uruguay.}
\begin{abstract}
We show that $C^r$ generically in the space of $C^r$ conservative diffeomorphisms of a compact surface, every hyperbolic periodic point has a transverse homoclinic orbit
\end{abstract}

\maketitle
\tableofcontents
\section{Introduction}

The existence of the so called \textit{homoclinic orbits} (or biasymtotic solutions) was detected by H. Poincaré in the study of the restricted three-body problem \cite{Po}. He was amazed by the dynamical complexity. Indeed, he stated \textit{"Rien n'est plus propre à nous donner une idée de la complication du problème des trois corps et en général de tous les problèmes de dynamique …"}. And about the picture generated by the presence of homoclinic orbits Poincaré also said \textit{"On sera frappé de la complexité de cette figure, que je ne cherche même pas à tracer."}. It was not until Smale and his creation of the horseshoe \cite{S} that a full comprehension of the dynamical complexity of the presence of a (transverse) homoclinic orbit was achieved. Afterwards, the question \textit{how common is the presence of a transverse homoclinic in the set of dynamical systems} have drawn the attention of many researchers.
For instance, Takens \cite{T} proved that $C^1$-generically in the space of $C^1$ conservative diffeomorphisms of a compact manifold, every hyperbolic periodic point has a transverse homoclinic orbit (for the dissipative case see \cite{C}). This results uses dramatically perturbation techniques available in the $C^1$ category, as the $C^1$ Closing Lemma of Pugh (\cite{Pu}). These techniques are completely unknown in the $C^r$ category, if $r\geq 2$. Our goal is to treat this case for surface dynamics.  Our main result is:

\begin{maintheorem}\label{t.mainintro}
For any compact boundaryless orientable surface $S$ and $1\leq r\leq \infty$, $C^r$-generically  in the space of $C^r$ conservative diffeomorphisms it holds that
there exist hyperbolic periodic points, and every such point has a transverse homoclinic intersection.
\end{maintheorem}

 We recall that in the case of the sphere, this theorem was proved  by Pixton \cite{P} (using ideas by Robinson \cite{R2}).  The same result was proved by Oliveira \cite{O1} for the torus.  None of the authors actually proved the generic existence of hyperbolic periodic points but, as we will see soon in this introduction, such a property is not difficult to get if $g\leq 1$ (see also Weiss \cite{W}). Later, Oliveira \cite{O2} proved Theorem  \ref{t.mainintro}  in higher genus, as soon as the action of $f$ on the first group of homology is irreducible. Finally, the fact that generically in the space of Hamiltonian diffeomorphisms of class $C^r$, every hyperbolic periodic orbit  has a transverse homoclinic intersection has been announced by Xia \cite{X2} (see also Theorem \ref{t.numberhyp}).

One could say that the proof of our main result is divided into two parts. On one hand, under some explicit $C^r$ generic assumptions,  we prove that if there are enough number of hyperbolic periodic points (i.e. larger than $2g-2$ where $g$ is the genus of the surface $S$) then \textit{every} hyperbolic periodic point has a transverse homoclinic orbit. On the other hand, there are examples where such condition on the number of hyperbolic periodic points fails. Here it comes an argument with perturbation flavour: we have to show that these examples can be $C^r$ perturbed in order to satisfy the above condition. To do this, we show that these situations can happen only  in the following cases:

\begin{itemize}
\item $g=1$ and the map is isotopic to the identity or to a power of a Dehn twist map,
\item $g>1$ and a power of the map  is isotopic to the identity.
\end{itemize}

The  former can be handled with a result by Addas-Zanata (\cite{Ad}). To deal with the latter, with the help of \textit{transverse foliations} developed by the first author (\cite{Lec1}), we show that the way how the stable and unstable manifolds wrap around the surface is so rigid that can be easily destroyed by a $C^r$ perturbation, by composing with a local rotation around  a special loop.   In the next subsection we will be more precise.

\subsection{Precise statements}\label{ss.statements}

In this article, $S$ will denote a smooth compact boundaryless orientable surface of genus $g$, furnished with a smooth area form $\omega$.  For $1\leq r\leq \infty$, denote $\Diff^r_{\omega}(S)$ the set of $C^r$ diffeomorphisms of $S$ preserving $\omega$ endowed with the $C^r$-topology. Recall that for a $C^r$ diffeomorphism $f:S\to S$ and a hyperbolic periodic point $p$, the  stable and unstable sets are $C^r$ injectively immersed manifolds. A (transverse) homoclinic point associated to a hyperbolic periodic point $p$ is a point (different of $p$) of (transverse) intersection between the stable and unstable manifolds of $p$.
We just restate Theorem 1:

\begin{theorem}\label{t.main}
 For $1\leq r\leq \infty$, there exists a residual set $\cR_r\subset\Diff^r_{\omega}(S)$ such that if $f\in \cR_r$, there exist hyperbolic periodic points, and every such point has a transverse homoclinic intersection.
\end{theorem}

Using the fact that the existence of a hyperbolic periodic point with a transverse homoclinic intersection is an open property that implies the positiveness of the entropy, we immediately deduce:

\begin{corollary}\label{l.entropy}
 For $1\leq r\leq \infty$, there exists a dense open set $\cO_r\subset\Diff^r_{\omega}(S)$ such that the topological entropy of every element of $ \cO_r$ is positive.\end{corollary}

Note that, by a well-known result of Katok \cite{K}, in the case of surface diffeomorphims of class $C^r$, $r\geq 2$, the positiveness of the entropy is equivalent to the existence of a transverse homoclinic intersection. Consequently, if $2\leq r\leq +\infty$, the previous corollary tells us that the set of $f\in \mathrm{Diff}^r_{\omega}(S)$  with positive topological entropy is open and dense. We will give now a more precise statement by breaking down the statement of Theorem \ref{t.main}.
Let $f$ be a diffeomorphism in $\Diff_\omega^r(S)$ and let $z$ be a periodic point of period $q$ \footnote{In the whole article, period will mean smallest period.}. We say that $z$ is {\it elliptic} if $Df^q(z)$ has two non real complex eigenvalues of modulus one, and we say that it is {\it hyperbolic} if $Df^q(z)$ has two real eigenvalues of modulus different from one. An {\it unstable branch} is a connected component of $W^u(z)\setminus\{z\}$ and a {\it stable branch} a connected component of $W^s(z)\setminus\{z\}.$ We refer as a branch any of the stable or unstable branches.
We first explicit the generic conditions we are going to use. We denote $\cG^r_{\omega}(S)\subset\Diff^r_\omega(S)$ the set of diffeomorphisms satisfying the following conditions.

\begin{itemize}

\item[(G1):]\label{n.genericity} Every periodic point is either elliptic or hyperbolic. Moreover, if $z$ is an elliptic periodic point of period $q$, then the eigenvalues of $Df^q(z)$ are not roots of unity.

 \item[(G2):]\label{h.genericity}  Stable and unstable branches of hyperbolic points that intersect must also intersect transversally  (in particular there is no saddle connection).
  \item[(G3):]\label{e.genericity} If $U$ is a neighborhood of an elliptic periodic point $z$, then there
is a topological closed disk $D$ containing $z$, contained in $U$, and bordered by finitely
many pieces of stable and unstable manifolds of some hyperbolic periodic
point $z'$.
\end{itemize}

Robinson \cite{R1} proved that, for any $r \ge 1$, properties (G1) and (G2) are $C^r$-generic  (it is easy to see that the no unity root condition is generic among elliptic periodic points), and (G3) is $C^r$-generic due to Zehnder \cite{Z}. Thus $\cG^r_{\omega}(S)$ is residual in $\Diff^r_\omega(S).$ The following simple facts will often be used in the article for a map $f\in\Diff^r_\omega(S)$:

\begin{itemize}
\item  if $q\not=0$, then $f\in \cG^r_{\omega}(S)$ if and only if $f^q\in\cG^r_{\omega}(S)$;
\item if $\pi:S'\to S$ is a finite covering and $f'$ a lift of $f$ to $S'$, then $f\in \cG^r_{\omega}(S)$ if and only if $f'\in\cG^r_{\pi^*(\omega)}(S')$.\end{itemize}

Denote  $\mathrm{fix}_h(f)$ the set of hyperbolic fixed points of $f\in \Diff^r_\omega(S)$ and $\mathrm{per}_h(f)$ the set of hyperbolic periodic points.
\medskip

The first result is folklore, consequence of Lefschetz formula:

\begin{proposition}\label{p.numberhyp} If $f\in \cG^r_{\omega}(S)$, then $\# \mathrm{per}_h(f)\geq \max(0,2g-2)$. Moreover, the inequality is strict if there exists at least one elliptic periodic point.
\end{proposition}

The second result asserts that Theorem \ref{t.main} is true if the number of hyperbolic points is greater that this lower bound.

\begin{theorem}\label{t.numberhyp} If $f\in \cG^r_{\omega}(S)$ and $\# \mathrm{per}_h(f)>\max(0,2g-2)$, then every hyperbolic periodic point of $f$ has a transverse homoclinic intersection.

\end{theorem}

 Note that Theorem \ref{t.numberhyp} already implies Theorem \ref{t.main} in the case of Hamiltonian diffeomorphisms, because every Hamiltonian diffeomorphism $f$ that belongs to $\cG^r_{\omega}(S)$ satisfies $\# \mathrm{per}_h(f)\geq 2g$, as a consequence of Arnold's conjecture  (see Floer \cite{Fl} or Sikorav \cite{Si}).

It remains to study the properties of diffeomorphisms $f\in \cG^r_{\omega}(S)$ such that $\# \mathrm{per}_h(f)=\max(0,2g-2)$. The proof of Theorem \ref{t.numberhyp} will give us:

\begin{proposition}\label{p.transitivity}
If $f\in \cG^r_{\omega}(S)$ and $\# \mathrm{per}_h(f)=\max(0, 2g-2)$ then

\begin{itemize}
\item $g\not =0$;

\item every periodic point is hyperbolic;

\item $f$ is transitive;

\item every stable or unstable branch of a hyperbolic periodic point of $f$ is dense

\item two different branches do not intersect.

\end{itemize}
\end{proposition}

The next result concerns isotopy classes of such diffeomorphisms:

\begin{theorem}\label{t.numberminhyp}
If $f\in \cG^r_{\omega}(S)$ is such that $\# \mathrm{per}_h(f)=\max(0, 2g-2)$ then:
\begin{itemize}
\item either  $g>1$ and there exists $q\geq 1$ such that $f^q$ is isotopic to the identity;
\item or $g=1$ and $f$ is  isotopic to a power of a Dehn twist map, meaning a homeomorphism conjugated to an automorphism $(x,y)\mapsto (x+my, y)$, $m\in\Z$.
\end{itemize}
\end{theorem}

The fact that $g\not=0$ is obvious. Indeed, every orientation preserving diffeomorphism $f$ of the sphere contains a fixed point. So, if $f$ satisfies $(G3)$, it  contains a hyperbolic periodic point.

Suppose now that $g=1$. One of the following situation occurs

\begin{enumerate}
\item $f$ is isotopic to a hyperbolic torus automorphism;

\item  $f$ is isotopic to a non trivial periodic torus automorphism;

\item  $f$ is isotopic to a non trivial power of a Dehn twist  map;

 \item $f$ is isotopic to the identity.

\end{enumerate}
The diffeomorphism $f$ cannot be isotopic to a hyperbolic torus automorphisms, otherwise it would have infinitely many periodic points. It cannot be isotopic to a non trivial periodic torus automorphism, otherwise it would have a fixed point (by Lefschetz formula). So $f$ is isotopic to a Dehn twist  map or to the identity. The proof of Theorem \ref{t.numberminhyp} will exclusively concern surfaces of genus greater than $1$.

To get Theorem \ref{t.main}, it remains to prove that a diffeomorphism $f\in \cG^r_{\omega}(S)$ such that $\# \mathrm{per}_h(f)=\max(0,2g-2)$ can be perturbed into a diffeomorphim $f'\in \cG^r_{\omega}(S)$ such that $\# \mathrm{per}_h(f')>\max(0,2g-2)$. An example of a diffeomorphism satisfying the hypothesis of Theorem \ref{t.numberminhyp} and the second conclusion is the map  $(x,y)\mapsto (x+my, y+\rho)$, where $\rho$ is irrational. How to perturb such a map to obtain a periodic orbit has been done by Addas-Zanata \cite{Ad}:

\begin{theorem}\label{t.dehn}
If $g=1$ and if $f\in \Diff^r_{\omega}(S)$ is isotopic to a non trivial power of a Dehn twist  map, then there exists $f'\in \cG^r_{\omega}(S)$, arbitrarily close to $f$, such that $\mathrm{per}_h(f')\not=\emptyset$.
\end{theorem}

So, to get Theorem \ref{t.main}, it remains to prove:

\begin{theorem}\label{t.isotopy}
If $f\in \Diff^r_{\omega}(S)$ and if there exists $q\geq 1$ such that $f^q$ is isotopic to the identity, then there exists $f'\in \cG^r_{\omega}(S)$, arbitrarily close to $f$, such that $\# \mathrm{per}_h(f')>\max(0, 2g-2)$.
\end{theorem}

Here again, suppose that $g=1$. As explained above, one can suppose that $q=1$ otherwise $f$ itself would have a fixed point.  The result is well known and related to the {\it rotation vector } $\mathrm{rot}_{f}(\mu_{\omega})\in \R^2/\Z^2$ (we will recall later the definition). The Conley-Zehnder theorem \cite{CZ} implies that $f$ has a periodic orbit  if $\mathrm{rot}_{f}(\mu_{\omega})\in \Q^2/\Z^2$.  So, one can compose $f$ with a rotation arbitrarily small to get a diffeomorphism $f''\in \mathrm{Diff}^r_{\omega}(S)$ such that $\mathrm{rot}_{f''}(\mu_{\omega})\in \Q^2/\Z^2$ and then approximate $f''$ by an element $f'$ of  $\mathcal{G}^r_{\omega}(S)$ having a periodic orbit. It must be noted that Addas-Zanata's theorem uses a similar argument. One can define the {\it vertical rotation number}, which belongs to $\R/\Z$ and the key result is that an element of $\mathrm{Diff}^r_{\omega}(S)$ isotopic to a non trivial power of the Dehn twist has a periodic orbit if its vertical rotation number is rational. Here again,  the proof of Theorem \ref{t.isotopy} will exclusively concern surfaces of genus greater than $1$.
The canonical example of a diffeomorphism $f\in \cG^r_{\omega}(S)$ such that $\# \mathrm{per}_h(f)=2g-2$ is given by the time one map of the flow of a minimal direction for a translation surface in the principal stratum.

\medskip

To conclude, recall that generically in $\Diff_\omega^r(S)$, the union of stable (or unstable) branches is dense (see \cite{FrLec}, \cite{X1}, \cite{KLecN}) and  that generically in the space of Hamiltonian diffeomorphisms of class $C^r$, periodic points are dense (Asaoka-Irie \cite{AsI}).
\medskip

\textit{Organization of the  article:} In Section \ref{s.preliminaries} we recall some notions and tools (as rotation vectors, transverse foliations, regular domains and the Nielsen-Thurston Classification Theorem) and some results regarding them which are important in our context. In Section \ref{s:An equivalence relation on the set of hyperbolic periodic points} we state an equivalence relation among periodic points. Proposition \ref{p.numberhyp} as well as a local version for  equivalent classes of periodic points are proved in Section \ref{s:minoration}. We give a criteria for the existence of a homoclinic class in Section \ref{s.criteria}. Theorem \ref{t.numberhyp} is proved in Section \ref{s:proof}. In Section \ref{s.dehn} we study  diffeomorphisms isotopic to Dehn twist maps and in Section \ref{s.numberminhyp} we prove Theorem \ref{t.numberminhyp} by proving that a generic  diffeomorphisms isotopic to a  Dehn twist map has infinitely many periodic points. Section \ref{s:dynamics-isotopic-identity} is devoted to the study of generic maps isotopic to the identity with a minimal number of periodic points. Theorem \ref{t.isotopy} (and hence Theorem \ref{t.main}) are proved in Section \ref{s:proof-isotopic-identity}. We give an outline of an alternate and  previous proof of Theorem \ref{t.isotopy} using Forcing Theory and a result by Lellouch \cite{Lel} in the final Section \ref{s:alternate}.

\section{Preliminaries}\label{s.preliminaries}

In this section we will state and develop some concepts, results and tools which are basics along the paper.

\subsection{Lefschetz index}\label{ss.lefschtez}

Let $f:S\to S$ be a diffeomorphism in $\mathcal{G}^r_\omega(S)$. The \textit{Lefschetz index} of a fixed point $i(f,z)$ is:
\begin{itemize}
\item equal to $1$ if $z$ is elliptic;
\item equal to $-1$ if $z$ is hyperbolic and its branches are fixed by $f$;
\item equal to $1$ if $z$ is hyperbolic and its branches are fixed by $f^2$ but not by $f$.


\end{itemize}


Lefschetz formula says that
$$\sum_{z\in \mathrm{fix}(f)}i(f,z)=\sum_{i=0}^2(-1)^i\mathrm{tr}(f_{*,i})=2-\mathrm{tr}(f_{*,1}),$$
where $f_{*,i}$  is the endomorphism of the $i$-th homology group $H_i(S,\R)$ induced by $f.$ Therefore, if $f$ is isotopic to the identity,
$$\#\mathrm{fix}_\mathrm{h}(f)\ge- \sum_{z\in \mathrm{fix}(f)}i(f,z)=\mathrm{tr}(f_{*,1})-2=2g-2,$$
with strict inequality in the presence of an elliptic fixed point or a hyperbolic fixed point with negative eigenvalues. In particular if $\#\mathrm{fix}(f)=2g-2$ then they are all hyperbolic. Moreover, if $\#\mathrm{per}_h(f)=2g-2$ then, they are all fixed.

\subsection{Regular domains and generic conservative diffeomorphisms}\label{ss.regular}

A {\it regular domain} of $S$ is a connected open set $V$ of finite type whose complement has no isolated point. Note that $V$ has finitely many ends and that its complement has finitely many connected components, none of them reduced to a point. Conversely, if $K$ is the union of finitely many connected closed sets, none of them reduced to a point, then every connected component of $S\setminus K$ is regular. Observe also that every connected component of the intersection of two (or finitely many) regular domains is a regular domain. It is a consequence  of Mayer Vietoris sequence,
 or of the previous characterisation of the complement of a regular domain. If $V$ is a regular domain, it can be compactified in three natural ways.
\begin{itemize}

\item The {\it ambient compactification} is the closure $\overline V$ of $V$ in $S$.

\item The {\it end compactification} is obtained by adding every end of $V$: one gets a boundaryless compact surface $\check V$.
 \item
The {\it prime end compactification} is obtained by blowing up every end by the circle of prime ends (see Mather \cite{Ma}): one gets a compact surface with boundary $\widehat V$.

\end{itemize}

If $V$ is invariant by an orientation preserving homeomorphism of $f$, then $f_{\vert \overline V}$ is an extension of $f_{\vert V}$ to $\overline V$.  Moreover there exists a natural extension $\check f$ of $f_{\vert V}$ to $\check V$ that permutes the ends. An important property of the prime end compactification is that $f_{\vert V}$ admits an extension by a homeomorphism $\widehat f$ of $\widehat V$: every added circle $C$ is periodic and if $q$ is its period, then $f^{q}{}_{\vert C}$ is orientation preserving, so the rotation number $\mathrm{rot}(C)\in\R/\Z$ of $f^{q}{}_{\vert C}$ can be defined if $C$ is endowed with the induced orientation.

\medskip

 The following proposition, due to Mather \cite{Ma}, is the key result of this article.

\begin{theorem}\label{t.pe-genericconsequences}
If $V$ is a regular domain invariant by $f\in \cG^r_{\omega}(S)$, then $\widehat f$ has no periodic point on the boundary of $\widehat V$. Equivalenty, for every added circle $C$, one has $\mathrm{rot}(C)\not\in \Q/\Z$.
\end{theorem}
The proof of the theorem was using a slightly different condition than (G3) but was extended to our situation in \cite{KLecN}. In the same article the following was shown (see  \cite[Theorem E]{KLecN}):

\begin{theorem}\label{t.genericconsequences}
If $V$ is a regular domain invariant by $f\in \cG^r_{\omega}(S)$, then $f$ has no periodic point on the frontier of $V$ in $S$.
\end{theorem}

An immediate consequence of Theorem \ref{t.pe-genericconsequences} is the following:

\begin{corollary}\label{c.e-genericconsequences}
Let $V$ be a regular open set invariant by $f\in \cG^r_{\omega}(S)$. If $\zeta$ is an end of $V$ of period $q$ (as a periodic point of $\check f)$, then for every $n\geq 1$, the Lefschetz index of   $\check f^{qn}$ at $\zeta$ is equal to $1$.
\end{corollary}

Note that if there exists $q\geq 2$ such that $f^q(V)=V$ (we will say that $V$ is periodic), we can apply the previous result to $f^q$ because $f^q\in \cG^r_{\omega}(S)$ if $f\in \cG^r_{\omega}(S)$.

Let us conclude this subsection with the following result, proved by Mather \cite{Ma} and extended in \cite[Corollary 8.9]{KLecN}
\begin{theorem}\label{t.genericconsequences2} The four branches of a hyperbolic periodic point $z$ of a diffeomorphism $f\in \cG^r_{\omega}(S)$ accumulate on $z$ and
 have the same closure in $S$.
\end{theorem}

\subsection{Rotation vectors}\label{ss.rotation}

The Poincaré Theory on rotation number of circle maps has been extended in a certain way to higher dimension (see Schwartzman \cite{Sc} for the seminal article). It is well known that not every point has a definite rotation vector. However, it can be defined for some of them. Let $S$ be a compact surface (with or without boundary), $f:S\to S$ a homeomorphism isotopic to the identity, and consider an isotopy $I=(f_t)_{t\in[0,1]}$ from the identity to $f$. Denote by $I(z)$  the {\it trajectory} of a point $z\in S$ as being the path $I(z):t\mapsto f_t(z)$ and by $I^n(z)$ the concatenation of the paths $I(z)$, \dots, $I(f^{n-1}(z))$, resulting in a path joining $z$ to $f^n(z).$ Furnish $S$ with a Riemannian metric  and for any pair of points $x,y\in S$ choose a path $c(x,y)$ joining $x$ to $y$ with uniformly bounded length (for instance take a minimal geodesic arc joining $x$ to $y$). Consider $z\in S$  and denote by $[\gamma_n(z)]\in H_1(S,\R)$ the class of the loop $\gamma_n(z)$ obtained by concatenation of  $I^n(z)$ and $c(f^n(z),z)$. Then, if the following limit exists, it is called the rotation vector of $z:$
$$\mathrm{rot}_I(z)=\lim_n\frac{[\gamma_n(z)]}{n}\in H_1(S,\R).$$
Note that it is independent of the choice of the Riemannian metric.  If $S$ is the compact annulus $\A=\R/\Z\times[0,1]$ or the torus $\R^2/\Z^2$, the rotation vector may depend on the isotopy. Nevertheless, in these cases, if $I'$ is another isotopy from the identity to $f$ then $\mathrm{rot}_I(z)-\mathrm{rot}_{I'}(z)$ belongs to $\Z$ in the case of the annulus and to  $\Z^2$ in the case of the torus. In the first case we can define $\mathrm{rot}_f(z)=\mathrm{rot}_I(z)+\Z\in\R/\Z$ and in the second case we can define $\mathrm{rot}_f(z)=\mathrm{rot}_I(z)+\Z^2\in\R^2/\Z^2$. If $S$ is different from $\A$ or $\R^2/\Z^2$, then all isotopies are homotopic (or $H_1(S,\R)=0$ if $S$ is the $2$-sphere), and so the rotation vector does not depend on the isotopy but only on $f$ and we write it $\mathrm{rot}_f(z).$

We now recall a well known result on the annulus, due to Birkhoff, extending the classical Poincaré-Birkhoff Theorem (see \cite{Bi}). Let $f:\A\to \A$ be a homeomorphism isotopic to the identity (i.e. that preserves the orientation and the boundary components). We say that $f$ satisfies the \textit{intersection property} if for every essential simple loop $\lambda$ (which means non homotopic to zero) we have $f(\lambda)\cap\lambda\neq\emptyset.$ We say that $f$ satisfies the \textit{boundary twist condition} if $\rho_0=\mathrm{rot}_I(z)\neq\mathrm{rot}_I(z')=\rho_1$ for $z\in \R/\Z\times\{0\}$ and $z'\in\R/\Z\times\{1\}$ for some (and hence for every) isotopy $I$ from $\mathrm{Id}$  to $f.$

\begin{theorem}\label{t.PBtheorem}
Let $f:\A\to \A$ be a homeomorphism isotopic to the identity satisfying the boundary twist condition and the intersection property and $I$ an isotopy from $\mathrm{Id}$ to $f$.Then, for every rational $p/q$ between $\rho_0$ and $\rho_1$ there exists a periodic point $z$ such that $\mathrm{rot}_I(z)=p/q.$
\end{theorem}

We turn our attention to \textit{rotation vectors of measures}.
Let $f$ be a homeomorphism of $S$ isotopic to the identity and ${\mathcal M}(f)$ the set of Borel probability measures that are invariant by $f$.   For every $\mu\in {\mathcal M}(f)$, we can define the {\it rotation vector} of $\mu$, it is an element $\mathrm{rot}_I(\mu)\in H_1(M,\R)$ defined as follows. If $\omega$ is a $C^1$ closed $1$-form, the integral $\int_{I(z)} \omega$ is well defined (even if the trajectory is not $C^1$) and depends continuously on $z$. The map
 $$\eta\mapsto \int_S\left(\int_{I(z)}\eta\right)\, d\mu(z)$$
 is a linear form on the space of $1$-forms that vanishes on the set of exact forms (because $\mu$ is invariant).  So it defines a linear form on the space $H^1(S,\R)$. Consequently, there exists $\mathrm{rot}_I(\mu)\in H_1(S,\R)$ such that for every $1$-form $\eta$, one has
 $$\int_S\left(\int_{I(z)}\eta\right)\, d\mu(z)=\langle [\eta], \mathrm{rot}_I(\mu)\rangle,$$
 where $[\eta]\in H^1(S,\R)$ is the cohomology class of $\eta$ and $\langle\enskip,\enskip\rangle$ the canonical form defined on $H^1(S,\R)\times H_1(S,\R)$. As a consequence of Birkhoff Ergodic Theorem, we can prove that $\mu$-almost every point has a rotation vector and that $$\mathrm{rot}_I(\mu)=\int_S \mathrm{rot}_I(z)\,d\mu.$$ We remark that if $\mu$ is an ergodic measure and $z$ is $\mu$-generic then $\mathrm{rot}_I(\mu)=\mathrm{rot}_I(z).$
In the case where $S=\A$, then $\mathrm{rot}_I(\mu)+\Z\in\R/\Z$ is independent of $I$, we denote it $\mathrm{rot}_f(\mu)$. Similarly, in the case where $S=\R^2/\Z^2$, then $\mathrm{rot}_I(\mu)+\Z^2\in\R^2/\Z^2$ is independent of $I$ and is denoted $\mathrm{rot}_f(\mu)$.  In the other cases, $\mathrm{rot}_I(\mu)$ does not depend on $I$, we will write $\mathrm{rot}_I(\mu)=\mathrm{rot}_f(\mu)$.

 Suppose that $\omega$ is a smooth volume form such that the associated measure $\mu_{\omega}$ is a probability measure. We will say that $f$ is {\it Hamiltonian} if $f$ preserves $\mu_{\omega}$ and if $\mathrm{rot}_f(\mu_{\omega})=0$. In case $f$ is a diffeomorphism of class $C^1$, it  is equivalent to say that $f$ is the time-one map of a time dependent Hamiltonian vector field ($\omega$ being considered as a symplectic form).

In what follows, $S$ is a smooth compact boundaryless orientable surface of genus $g\geq 2$. The following results are important for our goal:

 \begin{proposition} \label{prop:periodic}  We suppose that $f \in \mathrm{Diff}^1(S)$ preserves a Borel probability measure $\mu$ and that there exists $q\geq 2$ such that $f^q$ is isotopic to the identity. Then for every $\rho\in H_1(M,\R)$ we have:
$$f_{*,1}( \rho)\wedge \mathrm{rot}_{f^q}(\mu)= \rho\wedge \mathrm{rot}_{f^q}(\mu).$$
\end{proposition}

\begin{proof} For every $\rho\in H_1(M,\R)$, we will denote $\rho^{\perp}$ the space orthogonal to $\rho$  (relative to the intersection form $\wedge$). Recall that $\wedge$ is a symplectic form and $f_{*,1}$ a symplectic automorphism of $ H_1(M,\R)$.

We can suppose that $\mathrm{rot}_{f^q}(\mu)\not=0$, otherwise the conclusion is obvious. Let us prove first that $$f_{*,1}( \mathrm{rot}_{f^q}(\mu))= \mathrm{rot}_{f^q}(\mu).$$ Indeed, if $I=(f_t)_{t\in[0,1]}$ is an isotopy from $\mathrm{Id}$ to $f^q$, then $f\circ I\circ f^{-1}= (f\circ f_t\circ f^{-1})_{t\in[0,1]}$ is another isotopy from $\mathrm{Id}$ to $f^q$. Consequently, if $\eta$ is a closed $1$-form, we have\footnote{In the following calculations we assume that $f$ is of class $C^1.$ For a homeomorphism the calculations can be carried out as well with extra formalism. For the sake of simplicity we leave it in this form.}
$$\begin{aligned}
\langle [\eta], \mathrm{rot}_{f^q}(\mu)\rangle&= \int_S\left(\int_{f\circ I\circ f^{-1}(z)}\eta\right)\, d\mu(z)\\
&= \int_S\left(\int_{ I(f^{-1}(z))}f^*\eta\right)\, d\mu(z)\\
&= \int_S\left(\int_{ I(z)}f^*\eta\right)\, d(f^{-1}_*(\mu))(z)\\
&= \int_S\left(\int_{ I(z)}f^*\eta\right)\, d\mu(z)\\
&= \langle [f^*\eta], \mathrm{rot}_{f^q}(\mu)\rangle\\
&= \langle [\eta], f_{*,1}(\mathrm{rot}_{f^q}(\mu))\rangle. \end{aligned}
$$

 We deduce that
$$f_{*,1}( \mathrm{rot}_{f^q}(\mu))= \mathrm{rot}_{f^q}(\mu)$$
  and then that $$f_{*,1}( \mathrm{rot}_{f^q}(\mu)^{\perp})= \mathrm{rot}_{f^q}(\mu)^{\perp}$$
 and so $f_{*,1}$ induces a homothety on the one dimensional quotient space $H_1(S,\R)/ \mathrm{rot}_{f^q}(\mu)^{\perp}$. The ratio $\lambda$ of this homothety  is an eigenvalue of $f_{*,1}$. It is a root of unity because $f_{*,1}^q=\mathrm{Id}$, but it is also a real number and so either $\lambda=1$ or $\lambda =-1$.

 Suppose that $\lambda =-1$. For every $\rho\in H_1(M,\R)$, we have
 $$\begin{aligned}
 -\rho\wedge \mathrm{rot}_{f^q}(\mu)&=f_{*,1}( \rho)\wedge \mathrm{rot}_{f^q}(\mu)\\&=f_{*,1}( \rho)\wedge f_{*,1}( \mathrm{rot}_{f^q}(\mu))\\ &=\rho\wedge \mathrm{rot}_{f^q}(\mu).
 \end{aligned}$$ 
which contradicts the assumption $\mathrm{rot}_{f^q}(\mu)\not=0$.
We deduce that $\lambda=1$ and consequently get the lemma.
\end{proof}

The following corollary will permit to control the rotation vector for a perturbation on $f$ to create periodic orbits.

 \begin{corollary} \label{cor:periodic} We keep the assumptions of Proposition \ref{prop:periodic} and  consider $h \in \mathrm{Diff}^1(S)$ isotopic to the identity and preserving $\mu$. Then we have
$$\mathrm{rot}_{(h\circ f)^q}(\mu)\wedge\mathrm{rot}_{f^q}(\mu)=  q\,\mathrm{rot}_{h}(\mu)\wedge\mathrm{rot}_{f^q}(\mu).$$
\end{corollary}

\begin{proof} By Proposition \ref{prop:periodic}, it is sufficient to prove that
$$\mathrm{rot}_{(f\circ h)^q}(\mu) =\mathrm{rot}_{f^q}(\mu)+\sum_{k=0}^{q-1} f_{*,1}^k(\mathrm{rot}_{h}(\mu)).$$
If $I=(f_t)_{t\in[0,1]}$ and $I'=(f'_t)_{t\in[0,1]}$ are two continuous paths in $\mathrm{Diff}^1(S)$ such that $f_1=f'_0$ we will define the path $I*I'=(h_t)_{t\in[0,1]}$, where
$h_t=f_{2t}$ if $0\leq t\leq 1/2$ and $h_t=f'_{2t-1}$ if $1/2\leq t\leq 1$.  Then we can define by induction $I_1*I_2*\dots *I_n=I_1*(I_2*\dots *I_n)$ when it has a sense. Let $I$ be an isotopy from $\mathrm{Id}$ to $f^{q}$ and $J$ an isotopy from $\mathrm{Id}$ to $h$. The path
$$K=I*_{0\leq k\leq q-1} \left((h\circ f)^{k}\circ J\circ f^{q-k}\right)$$
is an isotopy from $\mathrm{Id}$ to $(h\circ f)^{q}$. Consequently, if $\eta$ is a closed $1$-form, we have
$$\begin{aligned}
&\langle [\eta], \mathrm{rot}_{(h\circ f)^q}(\mu)\rangle=\\&= \int_S\left(\int_{K(z)}\eta\right)\, d\mu(z)\\
&= \int_S\left(\int_{I(z)}\eta\right)\, d\mu(z)+\sum_{k=0}^{q-1} \int_S\left(\int_{ (h\circ f)^{k}\circ J\circ  f ^{q-k}(z)}\eta\right)\, d\mu(z)\\
&= \int_S\left(\int_{I(z)}\eta\right)\, d\mu(z)+\sum_{k=0}^{q-1} \int_S\left(\int_{  J(f ^{q-k}(z))}((h\circ f)^*)^{k}(\eta)\right)\, d\mu(z)\\
&= \int_S\left(\int_{I(z)}\eta\right)\, d\mu(z)+\sum_{k=0}^{q-1} \int_S\left(\int_{ J(z) }((h\circ f)^*)^{k}(\eta)\right)\, d (f^{q-k}_*(\mu))(z)\\
&= \int_S\left(\int_{I(z)}\eta\right)\, d\mu(z)+\sum_{k=0}^{q-1} \int_S\left(\int_{ J(z) }((h\circ f)^*)^{k}(\eta)\right)\, d\mu(z)\\
&= \langle [\eta], \mathrm{rot}_{f^q}(\mu)\rangle+\sum_{k=0}^{q-1} \langle [((h\circ f)^*)^{k}(\eta)], \mathrm{rot}_{h}(\mu)\rangle\\
&= \langle [\eta], \mathrm{rot}_{f^q}(\mu)+\sum_{k=0}^{q}  (h_*\circ f_{*,1})^k(\mathrm{rot}_{h}(\mu))\rangle\\
&= \langle [\eta], \mathrm{rot}_{f^q}(\mu)+\sum_{k=0}^{q}   f_{*,1}^k(\mathrm{rot}_{h}(\mu))\rangle.
\end{aligned}
$$ \end{proof}

\subsection{Brouwer lines and transverse foliations}\label{ss.foliation}

Let $\widetilde{f}$ be an orientation preserving homeomorphism of the plane with a discrete set of fixed points. An oriented  line $\widetilde{\phi}$ in the plane is an oriented proper embedding of the real line, and its complement has two connected components, say the left one $L(\widetilde{\phi})$ and the right one $R(\widetilde{\phi})$.  It  is called a \textit{Brouwer line} for $\widetilde{f}$ provided
$$\widetilde{f}(\overline{L(\widetilde{\phi})})\subset L(\widetilde{\phi}))$$
or equivalently if
$$\widetilde f^{-1}(\overline {R(\widetilde \phi)})\subset R(\widetilde \phi).$$

We will generalize this notion by defining a {\it singular Brouwer line} of $\widetilde f$ to be an oriented line $\widetilde\phi$ such that $$\widetilde f(\overline {L(\widetilde \phi)})\subset L(\widetilde \phi)\cup (\mathrm{fix}(\widetilde f)\cap \widetilde\phi)$$ or equivalently such that
$$\widetilde f^{-1}(\overline {R(\widetilde \phi)})\subset R(\widetilde \phi)\cup (\mathrm{fix}(\widetilde f)\cap \widetilde\phi).$$

 Let $S$ be surface of genus greater than $1$ and let $f:S\to S$ be a homeomorphism isotopic to the identity with finitely many fixed points.  Every isotopy $I=(f_t)_{t\in[0,1]}$ from $\mathrm{Id}$ to $f$ can be lifted to the universal covering space $\widetilde S$ into an isotopy $I=(f_t)_{t\in[0,1]}$ such that $\widetilde f_0=\mathrm{Id}$. The map $\widetilde f_1$ is independent of the chosen isotopy, it is the {\it natural lift} of $f$ and we denote it $\widetilde f$.
Define $\displaystyle\mathrm{fix}(I)=\bigcap_{t\in[0,1]} \mathrm{fix}(f_t)$ and $\mathrm{dom}(I)=S\setminus\mathrm{fix}(I)$. By finiteness of the fixed point set of $f$, there exists an integer $r$ such that:
\begin{itemize}
\item there exists an isotopy $I$ from $\mathrm{Id}$ to $f$ such that  $\#\mathrm{fix}(I)=r$;
\item there exists no isotopy $I$ from $\mathrm{Id}$ to $f$ such that  $\#\mathrm{fix}(I)>r$.
\end{itemize}
An isotopy $I$ from $\mathrm{Id}$ to $f$ such that  $\#\mathrm{fix}(I)=r$ is an example of a {\it maximal isotopy} (see \cite{BCLer} or \cite{J}) and one knows by \cite{Lec1} that there exists a topological oriented singular foliation $\mathcal F$ on $S$ such that:

\begin{itemize}
\item the singular set $\mathrm{sing}(\mathcal F)$ coincides with $\mathrm{fix}(I)$;
\item for every $z\in \mathrm{dom}(I)$, the trajectory $I(z)$ is homotopic in  $\mathrm{dom}(I)$, relative to the ends, to a path $\gamma_z$ positively transverse to $\mathcal F$, which means locally crossing each leaf from the right to the left.
\end{itemize}

The path $\gamma_z$ is not uniquely defined. In fact, it is uniquely defined up to a natural relation of equivalence. We call it a {\it transverse trajectory}. The foliation $\mathcal F$ is said to be {\it transverse to $I$}, it is not uniquely defined. One can lift the isotopy $I_{\vert\mathrm{dom}(I)}$ to an identity isotopy $\check I=(\check f_t)_{t\in[0,1]}$ on the universal covering space $\check{\mathrm{dom}}(I)$ of $\mathrm{dom}(I)$. Similarly $\mathcal F_{\vert\mathrm{dom}(I)}$ can be lifted to a non singular foliation $\check{\mathcal F}$ and so each leaf $\check\phi$ of $\check{\mathcal F}$ is an oriented line. Fix $\check z\in\check{\mathrm{dom}}(I)$, denote $z$ the projection of $\check z$ in $\mathrm{dom}(I)$ and $\phi_{\check z}$ the leaf containing $\check z$. The path $\gamma_z$ defined by the theorem can be lifted to a path joining $\check z$ to $\check f_1(\widetilde z)$ and positively transverse to $\check {\mathcal F}$. We deduce that $\check f_1(\check z)$ belongs to  $L(\phi_{\check z})$ and we can prove similarly that $\check f_1{}^{-1}(\phi_{\check z})$ belongs to  $R(\phi_{\check z})$. Consequently each leaf of $\check{\mathcal F}$ is a Brouwer line of $\check f_1$.  In other words $\mathcal F$ is lifted to a foliation by Brouwer lines of $\check f_1$ (see \cite{Lec1}).

\begin{lemma}\label{l.unlinkedfixedpoints}
Assume that $f\in\mathcal{G}^r_\omega(S)$ is isotopic to the identity and that the number of fixed points of $f$ is $2g-2.$ Then $f$ is isotopic to the identity relative to its fixed point set.
\end{lemma}

\begin{proof}
Consider a maximal isotopy $I$ as above and a transverse foliation $\mathcal F$. An important property of transverse foliations (see \cite{Lec2}), stated in our context, says that if the Hopf index $i({\mathcal F},z)$ of $z\in \mathrm{fix}(I)$ is different from $1$, then it is equal to the Lefschetz index $i(f,z)$. So
$$i({\mathcal F},z)\not=1\Rightarrow i({\mathcal F},z)=i(f,z)=-1.$$ Applying Hopf formula, one gets
$$2-2g=\chi(S)=\sum_{z\in\mathrm{fix}(I)}i({\mathcal F},z)\geq -r,$$
and so $r=2g-2$.
\end{proof}
\begin{remark*}
As explained in Section \ref{ss.lefschtez}, every map $h\in{\mathcal G}^r_{\omega}(S)$ isotopic to the identity has at least  $2g-2$ fixed points. A slight modification of Lefschetz formula would tell us that there are at least  $2g-2$ fixed points that are lifted to fixed points of the natural lift $\widetilde h$. The proof above is an alternate proof of this fact but says more: $h$ is isotopic to the identity relative to a set containing at least $2g-2$ fixed points.  In our situation this set is nothing but   $\mathrm{fix}_h (f)$. Consequently, $I$ is uniquely defined, up to a homotopy relative to the fixed point set, we will denote it $I=I_f$.
\end{remark*}

Let us state now a perturbation result about transverse foliations. A {\it cycle of connections} of $\mathcal F$ is  a simple loop, union of finitely many singular points $(z_i)_{i\in\Z/r\Z}$ and finitely many leaves $(\phi_i)_{i\in\Z/r\Z}$, such that the $\alpha$-limit set of $\phi_i$ is equal to $z_i$ and its $\omega$-limit set is equal to $z_{i+1}$.

\begin{lemma} \label{l:loops}
Assume that $f\in\mathcal{G}^r_\omega(S)$ is isotopic to the identity and that the number of fixed points of $f$ is $2g-2.$ There exists a foliation transverse to $I_f$, that satisfies one of the two following properties:

\begin{itemize}
\item it admits  a closed leaf;
\item it admits a cycle of connections.
\end{itemize}

\end{lemma}

\begin{proof} Let $\mathcal F$ be a foliation transverse to $I_f$. As explained previously, the singular points are all saddle points of $\mathcal F$. One of the following situation occurs:
\begin{itemize}
\item  for every leaf $\phi$, there exist singular points $z$ and $z'$ (that could be equal), such that the $\alpha$-limit set of $\phi$ is equal to $z$ and its $\omega$-limit set is equal to $z'$;
\item there exists a non wandering leaf.
\end{itemize}
In the first case, there exists a cycle of connections because there are finitely many singular points.

In the second case, we need to use a perturbation argument, which is a very slight modification of a perturbation result of Le Roux \cite{Ler} in the local case.

Let $\phi$ be a non wandering leaf,  $z_0$ a point on $\phi$ and $U_1$ a flow box of $\mathcal F$ containing $z_0$. We will consider a chart $\psi:U_1 \to [-1,1]^2$, sending $z_0$ onto $(0,0)$ and $\mathcal F_{\vert U_1}$ onto the vertical foliation oriented downward. We define a horizontal path on $U_1$ as a path sent onto a horizontal path by $\psi$.  For every $\varepsilon\in(0,1]$, we define $U_{\varepsilon}=\psi^{-1}([-\varepsilon, \varepsilon]^2)$.

Let $\gamma_{z_0}$ be a transverse trajectory of $z_0$. It passes through $z_0$ finitely many times. The point $z_0$ being not fixed by $f$, one can choose $\gamma_{z_0}$  such that it does not pass through $z_0$ but at its initial end. Furthermore, one can suppose that there exists $\varepsilon_0>0$ such that $\gamma_{z_0}$ intersects $U_{\varepsilon_0}$ on the horizontal segment $\psi^{-1}([0, \varepsilon_0]\times\{0\})$. Moreover there exists a neighborhood $V_{z_0}\subset U_{\varepsilon_0}$ of $z_0$ such that for every $z\in V_{z_0}$, a transverse trajectory $\gamma_{z}$  can be chosen to intersect $U_{\varepsilon_0}$ exactly on a horizontal segment. Taking a smaller value of $\varepsilon_0$ if necessary, $\gamma_{f^{-1}(z_0)}$ can be chosen to intersect $U_{\varepsilon_0}$ exactly on the horizontal segment $\psi^{-1}([-\varepsilon_0,0]\times\{0\})$  and there exists a neighborhood $V_{f^{-1}(z_0)}\subset S\setminus\mathrm{fix}(f)$ of $f^{-1}(z_0)$ such that for every $z\in V_{f^{-1}(z_0)}$, the transverse trajectory $\gamma_{z}$ can be chosen to intersect $U_{\varepsilon_0}$ exactly on a horizontal segment.

Similarly, for every point $z\in S\setminus(\mathrm{fix}(f)\cup\{z_0,f^{-1}(z_0)\})$, there exists a neighborhood $V_{z}\subset S\setminus\mathrm{fix}(f)$ of $z$ and $\varepsilon_z\in(0,\varepsilon_0]$ such for every $z'\in V_{z}$, the transverse trajectory $\gamma_{z'}$ can be chosen not to intersect $U_{\varepsilon_z}$.

The dynamics of $\mathcal F$ is well understood in a neighborhood of a singular point $z$ (see \cite{Ler}). For every neighborhood $U$ of $z$ there exists a neighborhood $V$ of $z$ such that for every $z'\in V\setminus\{z\}$, the transverse trajectory $\gamma_{z'}$ can be chosen to lie inside $U$. Consequently, for every $z\in \mathrm{fix}(f)$, there exists a punctured neighborhood $V_{z}\subset S\setminus\mathrm{fix}(f)$ of $z$, such for every $z'\in V_{z}$, the transverse trajectory $\gamma_{z'}$ can be chosen not to intersect $U_{1}$.

One can cover the compact set $S\setminus\cup_{z\in\mathrm{fix}(f)\cup\{z_0,f^{-1}(z_0)\}} V_z$ by a finite family $(V_{z_j})_{j\in J}$, where $z_j\in S\setminus\mathrm{fix}(f)\cup\{z_0,f^{-1}(z_0)\}$. Setting $\varepsilon=\min_{j\in J} \varepsilon_{z_j}$, one gets the following property: for every $z\not\in\mathrm{fix}(f)$, there exists a choice of $\gamma_{z'}$ that does not intersect $U_{\varepsilon}$ or intersect it in a horizontal path.

One can modify the foliation in the interior of $U_{\varepsilon}$. If the leaves inside $U_{\varepsilon}$ are transverse to the horizontal foliation, the property above says that the new foliation is still transverse to $I_f$. The leaf $\phi$ being non wandering, such a modification can be done with a closed leaf in the new foliation.

\end{proof}

The next result is important for our needs.
\begin{proposition} \label{prop:Brouwerlines}
Let $S$ be a surface of genus larger than $1$ and let $f\in \mathcal{G}^r_w(S)$ be isotopic to the identity such that $\#\mathrm{fix}_h (f)=2g-2.$ Then,
there exists an oriented loop $\phi\subset S$ non homologous to zero that is lifted to a Brouwer line (possibly singular) of $\widetilde f$.
\end{proposition}

\begin{proof}

Let $\mathcal F$ be a foliation transverse to $I_f$ with a closed leaf or a cycle of connections $\phi$. This loop is not homologous to zero. Otherwise it would bound a surface $\Sigma$ on its left such that $f(\Sigma)\subset \mathrm{int}(\Sigma)\cup (\mathrm{fix}(f)\cap\phi)$. This would contradict the fact that $f$ preserves $\mu_{\omega}$.   One can lift the isotopy $I$ to an identity isotopy $\widetilde I=(\widetilde f_t)_{t\in[0,1]}$ on $\widetilde S$ joining $\mathrm{Id}$ to $\widetilde f$. Similarly $\mathcal F$ can be lifted to a singular foliation $\widetilde{\mathcal F}$. The leaf $\phi$ can be lifted to an oriented line  $\widetilde\phi$ which is either a leaf of $\widetilde{\mathcal F}$ or a union of leaves and singular points of $\widetilde{\mathcal F}$. In the first case, $\widetilde \phi$ is a Brouwer line; in the second case it is a singular Brouwer line. Indeed,  for every $\widetilde  z\in\widetilde{\mathrm{dom}}(I)$, the path $\gamma_z$ defined by the theorem joining the projection $z\in S$ of $\widetilde z$ to its image by $f$ can be lifted to a path joining $\widetilde z$ to $\widetilde f(\widetilde z)$ and positively transverse to $\widetilde {\mathcal F}$.  This concludes the proof.
\end{proof}

\subsection{Nielsen-Thurston Classification}\label{ss.NT}

\begin{definition}
A {\it Dehn twist map} of $S$ is an orientation preserving homeomorphism $h$ of $S$ that satisfies the following properties:

\begin{itemize}
\item there exists a non empty finite family $(A_{i})_{i\in I}$ of pairwise disjoint invariant essential closed annuli;

\item no connected component of $S\setminus \cup_{i\in I}A_i$ is an annulus;

\item $h$ fixes every point of $S\setminus \cup_{i\in I}A_i$;

\item for every $i\in I$, the map $h_{\vert A_i}$ is conjugate to $\tau^{n_i}$, $n_i\not=0$, where $\tau$ is the homeomorphism of $\T\times[0,1]$ that is lifted to the universal covering space by $\widetilde\tau:(x,y)\mapsto (x+y,y)$.

\end{itemize}

The annuli $A_i$ will be called  the {\it twisted annuli} and $n_i$ the {\it twist coefficients}.

\end{definition}

 By Thurston-Nielsen theory (see \cite{Th},\cite{CB}, \cite{FLP}), one knows that $f$ is isotopic to a homeomorphism $h$  such that

\begin{itemize}
\item there exists a finite family $(\lambda_{i})_{i\in I}$ of pairwise disjoint essential closed loops invariant by $h$;

\item no connected component of $S\setminus \cup_{i\in I}\lambda_i$ is an annulus;

\item the closure $C$ of such a component is invariant by a power $h^m$ of $h$ and $h^m{}_{\vert C}$ is either isotopic to a pseudo-Anosov map or a periodic map.

\end{itemize}
Consequently, one the following situation occurs:
\begin{enumerate}
\item there exists a pseudo-Anosov component;

\item  there is a power of $f$ that is isotopic to a Dehn twist map;

\item there is a power of $f$ that is isotopic to the identity.

\end{enumerate}

It is well known and folklore that the existence of a  pseudo-Anosov component implies the existence of infinitely many periodic points. This essentially follows from \cite{H} where he treats the case of a pseudo-Anosov map.

\section{An equivalence relation on the set of hyperbolic periodic points}\label{s:An equivalence relation on the set of hyperbolic periodic points}

We will say that two hyperbolic periodic points $z$ and $z'$ of
a map $f\in \cG^r_{\omega}(S)$  are {\it equivalent}, and we will write $z\sim z'$, if the branches of $z$ and the branches of $z'$ have the same closure. One gets an equivalence relation on $\mathrm{per}_h(f)$.  We will denote $\cE(f)$ the set of equivalence classes and for every $\kappa\in\cE(f)$, we will write $K(\kappa)$ for the closure of a branch of an element $z\in\kappa$. The branches of elements $z\in\kappa$ will be called {\it branches of $\kappa$}. The map $f$ acts naturally on $\cE(f)$ as a bijection and every orbit is finite (because $f^q(\kappa)=\kappa$ if $\kappa$ contains a fixed point of $f^q$) and so one can define the period of $\kappa$ as the cardinal of its orbit. Of course $\cE(f^q)=\cE(f)$, for every $q\geq 2$.

\medskip

Let us state two facts that will be needed frequently later.

\begin{proposition}\label{p.regulardomain-class}
If $V$ is a periodic regular domain of $f\in \cG^r_{\omega}(S)$, then every class $\kappa\in \cE(f)$ is included in $V$ or disjoint from $\overline V$. In the first situation the branches of $\kappa$ are all included in $V$, in the second situation they are all disjoint from $\overline V$.\end{proposition}

\begin{proof}  The frontier of $V$ does not contain any periodic point by Theorem \ref{t.genericconsequences}. Moreover it is periodic. Consequently it does not meet any branch. The branches of a class $\kappa\in \cE(f)$, being connected and accumulating on every point of this class, we deduce the proposition. \end{proof}

\begin{corollary}\label{c.genericconsequences}
Fix $f\in \cG^r_{\omega}(S)$, $\kappa\in \cE(f)$ and $z'\in\mathrm{per}_h(f)$. If $z'$ belongs to $K(\kappa)$,  then it belongs to $\kappa$. \end{corollary}

\begin{proof}  Denote $\kappa'$ the class of $z'$. Every connected component of $S\setminus K(\kappa')$ is a periodic regular domain. By Proposition \ref{p.regulardomain-class} one deduces that either $\kappa$ is included in one of the components of $S\setminus K(\kappa')$, or included in $K(\kappa')$. The first situation is impossible, because the branches of $\kappa$ would be included in this component and the point $z'$ would belong to the frontier of the component, because accumulated by these branches. This contradicts Theorem \ref{t.genericconsequences}.  In fact the same contradiction occurs if there is a branch of $\kappa$ that meets a connected component of $S\setminus K(\kappa')$ because its end will belong to the frontier of this component. Consequently, we have $\kappa\subset K(\kappa)\subset K(\kappa')$. Replacing $z'$ by any point $z$ of $\kappa$, we deduce similarly that  $\kappa'\subset K(\kappa')\subset K(\kappa)$. The equality $K(\kappa)=K(\kappa')$ tells us that $\kappa=\kappa'$.  \end{proof}

\begin{corollary}\label{c.separationconsequences}
If $V$ is a periodic regular domain of $f\in \cG^r_{\omega}(S)$, that contains $p$ different classes $\kappa_i$, $1\leq i\leq p$, one can find a family $(V_i)_{1\leq i\leq p}$ of pairwise disjoint periodic regular domains such that $\kappa_i\subset V_i\subset V$.  \end{corollary}

\begin{proof}  Fix $1\leq i<j\leq p$. Every connected component of $S\setminus K(\kappa_j)$ is a periodic regular domain and we know by
Corollary \ref{c.genericconsequences} that $\kappa_i\cap K(\kappa_j)=\emptyset$. So, by Proposition \ref{p.regulardomain-class}, there exists a connected component $W_i^j$ of $S\setminus K(\kappa_j)$ that contains $\kappa_i$. We define inductively a family $(W_i)_{1\leq i\leq p}$ of pairwise disjoint periodic regular domains such that $\kappa_i\subset W_i$ in the following way;
\begin{itemize}
\item  $W_1$ is the connected component of $\bigcap_{1< j\leq p} W_1^j$ that contains $\kappa_1$:
\item for every $i>1$,  $W_i$ is the connected component of $$ \left( \bigcap_{i<j\leq p} W_i^j\right)\setminus \left(\bigcup_{1\leq j<i} \overline{W_j}\right)$$ that contains $\kappa_i$.
\end{itemize}
To finish the proof, it is sufficient to define $V_i$ has being the connected component of $V\cap W_i$ that contains $\kappa_i$.
\end{proof}

Recall that the genus $g(V)$ of an open set $V\subset S$ is the largest integer $s$ such that we can find a family of simple loops $(\lambda_i)_{0\leq i<2s}$ satisfying:
\begin{itemize}
\item $\lambda_{2j}$ and $\lambda_{2j+1}$ intersect in a unique point;
\item $[\lambda_{2j}]\wedge[\lambda_{2j+1} ]=1$;
\item $\lambda_i\cap \lambda_{i'}=\emptyset$, if $i\not=i'$ and $\{i,i'\}$ is not a set $\{2j, 2j+1\}$.
\end{itemize}
Let us define now the {\it genus of class $\kappa \in\cE(f)$}, where $f\in\cG^r_{\omega}(S)$ as being the integer $g(\kappa)\in\{0,\dots, g\}$ uniquely defined by the following conditions:
\begin{enumerate}
\item $\kappa$ is contained in a periodic regular domain of genus $g(\kappa)$;
\item  $\kappa$ is not contained in a periodic regular domain of genus $<g(\kappa)$.
\end{enumerate}

The function $\kappa\mapsto g(\kappa)$ satisfies some additive properties. For example, in the statement of Corollary \ref{c.separationconsequences}, it holds:
$$\sum_{i=1}^pg(\kappa_i)\leq \sum_{i=1}^pg(V_i)\leq g(V).$$
 In particular, by Corollary \ref{c.separationconsequences}, if $V$ is a periodic regular domain of genus $g(\kappa)$ containing $\kappa$, then every class $\kappa'\subset V$ distinct from $\kappa$ has genus zero.  The following result is stronger:

\begin{proposition}\label{p.regulardomain-minimalgenus}
Let $\kappa \in\cE(f)$ be a fixed class of $f\in \cG^r_{\omega}(S)$ and $V$ a fixed regular domain of genus $g(\kappa)$ containing $\kappa$. Then, for every finite family  $(\kappa_i)_{i\in I}$ of classes  in $\cE(f)$ included in $V$ and distinct from $\kappa$, there exists a finite family $(\check D_j)_{j\in J}$ of periodic regular open disks of the end compactification $\check V$ of $V$ such that:

\begin{itemize}
\item the $\check D_j$, $j\in J$, are pairwise disjoint;
\item$\kappa\cap \overline{\check D_j}=\emptyset$, for every $j\in J$;
\item for every $i\in I$, there exists $j\in J$ such that $\kappa_i\subset \check D_j$;
\item for every $j\in J$, there exists $i\in I$ such that $\kappa_i\subset \check D_j$;
\item if $n_i$ is the period of $\kappa_i$ and if $\kappa_i\subset \check D_j$, then $\check f^{n_i}(\check D_j)=\check D_j$;
\item if  $\check f^{n_i}(\check D_j)=\check D_j$, one can extend $\check f^{n_i}{}_{\vert \check D_j}$ to the prime end compactification of $\check D_j$ by adding a circle with no periodic point.
\end{itemize}
\end{proposition}

\begin{proof} The set $\left(S\setminus K(\kappa)\right)\cap V$ is an open set invariant by $f$ whose connected components are regular. So, there exists a finite family $( V_j)_{j\in J}$ of periodic domains, connected components of $\left(S\setminus K(\kappa)\right)\cap V$ such that:

\begin{itemize}
\item the $V_j$, $j\in J$, are pairwise disjoint;

\item for every $i\in I$, there exists $j\in J$ such that $\kappa_i\subset V_j$;
\item for every $j\in J$, there exists $i\in I$ such that $\kappa_i\subset V_j$;
\item if $n_i$ is the period of $\kappa_i$ and if $\kappa_i\subset V_j$, then $f^{n_i}(V_j)=V_j$.
\end{itemize}

Moreover,  $\kappa\cap \overline{V_j}=\emptyset$, for every $j\in J$, and so the connected component $W$ of $V\setminus \left(\bigcup_{1\leq j<i} \overline{V_j}\right)$ that contains $\kappa$ is a  periodic regular domain of genus $g(\kappa)$, by definition of $g(\kappa)$. One deduces that every $V_j$ has genus $0$, it is a punctured disk.
The morphism $i_*: H_1(W,\R)\to H_1(\check V,\R)$ induced by the inclusion map $i:W\hookrightarrow \check V$ is onto, because $g(W)=g(\check V)$. So, for every $j\in J$,  the morphism $i'_*: H_1(V_j,\R)\to H_1(\check V,\R)$ induced by the inclusion map $i':V_j\hookrightarrow \check V$ is null , because its image is in the orthogonal of the image of $i_*$ for the intersection form $\wedge$. One deduces that the closure of $V_j$ in $\check V$ is a disk $\check D_j$, obtained by adding  the ends of $\overline V$ which are ends of $\check D_j$. The family $(\check D_j)_{j\in J}$ satisfies the conditions of the Proposition \ref{p.regulardomain-minimalgenus}.\end{proof}

\section{Minoration of the number of hyperbolic periodic points} \label{s:minoration}

Note that if $f\in \cG^r_{\omega}(S)$, then $\mathrm{fix}(f^n)$ is finite, for every $n\geq 1$. Moreover, the index  $i(f^n,z)$ can be easily computed as it was shown in Section \ref{ss.lefschtez}. We will begin by proving Proposition \ref{p.numberhyp}, which means proving that $\# \mathrm{per}_h(f)\geq 2g-2$ for every  $f\in \cG^r_{\omega}(S)$ with a strict inequality if there exists an elliptic periodic point.

\medskip
\begin{proof}[Proof of  Proposition \ref{p.numberhyp}.]  Applying Lefschetz formula to $f^n$, $n\geq1$, one gets
$$\sum_{z\in \mathrm{fix}(f^n)}i(f^n,z)= \sum_{i=0}^2 (-1)^i\mathrm{tr}( f_{*,i}^n) =2-\mathrm{tr}(f_{*,1}^n),$$
where $f_{*,i}$ denotes the endomorphism of the $i$-th homology group $H_i(S,\R)$ induced by $f$, and $ \mathrm{tr}( f_{*,i}^n)$ denotes the trace of $f_{*,i}^n$. Consequently, one gets
$$\# \mathrm{fix}_h(f^n)\geq -\sum_{z\in \mathrm{fix}(f^n)}i(f,z)= \mathrm{tr}(f_{*,1}^n)-2.$$
Denoting $(\xi_j)_{1\leq j\leq 2g}$ the complex eigenvalues of $f_{*,1}$, one knows that
$$ \mathrm{tr}(f_{*,1}^n)=\sum_{j=1}^{2g} \xi_j^n.$$
A rotation of $\R^{2g}/\Z^{2g}$ having only recurrent points, one can find $n$ arbitrarily large such that the argument of every $\xi_j^n$ belongs to $(-\pi/4,\pi/4)$. In that case, one gets
$$ \mathrm{tr}(f_{*,1}^n)\geq{\sqrt 2\over 2} \,\sum_{j=1}^{2g} \vert \xi_j\vert^n.$$
Consequently, if $\sup_{1\leq j\leq 2g} \vert \xi_j\vert>1$, then $\# \mathrm{per}_h(f)$ is infinite. Otherwise, every $\xi_j$ is on the unit circle because the determinant of $f_{*,1}$ is equal to $1$. Changing $\pi/4$ with $\delta$, arbitrarily small, one deduces that for every $\varepsilon>0$, there exists $n$ such that
$$ \mathrm{tr}(f_{*,1}^n)\geq(1-\varepsilon)\, \sum_{j=1}^{2g} \vert \xi_j\vert^n= (1-\varepsilon)2g.$$
The trace $\mathrm{tr}(f_{*,1}^n)$ being an integer, if $\varepsilon$ is chosen small enough, one deduces that
$$ \mathrm{tr}(f_{*,1}^n)=2g,$$
which means that $\xi_j^n=1$ for every $j$. Finally, one gets $\# \mathrm{fix}_h(f^n)\geq 2g-2$.
In this case,  $ \mathrm{tr}(f_{*,1}^{nq})=2g,$ for every $q\geq 1$.

Suppose now that $f$ has an elliptic periodic point of period $q$. One gets
$$\# \mathrm{fix}_h(f^{nq})> -\sum_{z\in \mathrm{fix}(f^{nq})}i(f^{nq},z)= \mathrm{tr}(f_{*,1}^{nq})-2 =2g-2.$$
\end{proof}

We will state now a local version of Proposition  \ref{p.numberhyp} related to a class $\kappa\in \cE(f)$.

\begin{proposition}\label{p.localLefschetz}
Let $f\in \cG^r_{\omega}(S)$. There are two possible situations:

\begin{itemize}
\item for every class $\kappa\in \cE(f)$ we have $\# \kappa>2g(\kappa)-2$;
\item there is no periodic regular domain but $S$, there is a unique class in  $\cE(f)$ and its cardinal is $2g-2$.
\end{itemize}
\end{proposition}

\begin{proof} To get the proposition, it is sufficient to prove that if $\kappa$ is a class of genus $g'$ included in a periodic regular domain $V$ of genus $g'$ and if $V\not=S$, then $\# \kappa>2g'-2$. Of course one can suppose that $\kappa$ is finite. Replacing $f$ with a power of it, one can suppose that $V$ is invariant by $f$, that every point $z$ of $\kappa$ is fixed by $f$, that every branch of $\kappa$ is fixed by $f$ and that every end of $V$ is fixed by $\check f$. We write $E$ for the set of ends. For every $n\geq 1$, the set $\mathrm{fix}(\check f^n)$ is finite. Moreover, as explained in Corollary \ref{c.e-genericconsequences}, one can compute the index $i(\check f^n,z)$ of $z\in \mathrm{fix}(\check f^n)$:  \begin{itemize}
 \item
 it is equal to $1$ is $z\in E$;
\item it is equal to $1$ if $z\in V$ is an elliptic fixed point of $f^n$;
\item it is equal to $1$ if $z\in V$ is a hyperbolic fixed point of $f^n$ and its branches are fixed by $f^{2n}$ but not by $f^n$;
\item
 it is equal to $-1$ is $z\in V$ is a hyperbolic fixed point of $f^n$ and its branches are fixed by $f^n$.

 \end{itemize}

Lefschetz formula tells us that
$$\sum_{z\in \mathrm{fix}(\check f^n)}i(\check f^n ,z)= 2-\mathrm{tr}(\check f_{*,1}^n),$$
and we have seen in the proof of Proposition \ref{p.numberhyp} that there exists $n$ such that
$$ \mathrm{tr}(\check f_{*,1}^n)\geq 2g'.$$

There are finitely many classes included in $V$ that contains a fixed point of $f^n$. By Proposition \ref{p.regulardomain-minimalgenus}, there exists a finite family $(\check D_j)_{j\in I}$ of pairwise disjoint regular open disks of $\check V$, invariant by $\check f^n$, whose union contains the hyperbolic fixed points of $f^n$ in $V$ that do not belong to $\kappa$. Write $E_1$ for the set of fixed points of $f^n$ of index one belonging to $V$ but not to a $\check D_j$, $j\in J$ (meaning elliptic or hyperbolic with reflexion) and $E_2$ the set of ends of $V$ that do not belong to any $\check D_j$, $j\in J$.
We deduce from Lefschetz formula the following inequality:
$$\# \kappa\geq 2g'-2+\#E_1+\#E_2+\sum_{j\in J}\left(\sum_{z\in \check D_j, \check f^n(z) =z} i(\check f^n, z)\right)$$
By  Proposition  \ref{p.regulardomain-minimalgenus},  $\check D_j$ can be compactified by adding a circle of prime ends with an irrational rotation number. Consequenty, one has $$\sum_{z\in \check D_j, \check f^n(z) =z} i(\check f^n, z)=1.$$ One deduces that
$$\# \kappa\geq 2g'-2+\#E_1+\#E_2+\#J.$$
It remains to say that if V is not equal to $S$ it contains an end and so, at least one of the sets $E_2$ or $J$ is not empty.
\end{proof}

\section{A criteria of existence of homoclinic classes}\label{s.criteria}
Fix $f\in \cG^r_{\omega}(S)$ and $\kappa\in\cE (f)$. Suppose that for some $z\in \kappa$ there is an unstable branch of $z$ that intersects a stable branch of $z.$ By (G2) we know that these branches intersect transversally. Using the classic $\lambda$-lemma and the fact that all the branches of $z$ has the same closure we conclude that any stable branch of $z$ intersects any unstable branch of $z.$ Moreover, since all branches of $\kappa$ have the same closure, we deduce that every stable branch of $\kappa$ intersect every unstable branch of $\kappa$.  We will say that $\kappa$ is a {\it homoclinic class}.

In particular, if for some $\kappa\in \cE (f)$ there is a family $(z_i)_{i\in \Z/r\Z}$ in $\kappa$ such that for every $i\in\Z/r\Z$, there is an unstable branch of $z_i$ and a stable branch of $z_{i+1}$ that intersect, using again (G2) and the $\lambda$-lemma, we get that $\kappa$ is a homoclinic class.

Let state the main result of this section

\begin{proposition}\label{p.cycle} Let $f\in \cG^r_{\omega}(S)$ and $V$ a regular domain of genus $g'$. Every set of hyperbolic periodic points in $V$, whose cardinal is larger than $2g'$, contains at least one point whose equivalence class is homoclinic.
\end{proposition}

\begin{proof}

We suppose that $(z_i)_{i\in I}$ is a finite family of distinct hyperbolic periodic points in $V$ and that $\# I>2g'$. We want to prove that there exists $i\in I$ such that $z_i$ has a transverse homoclinic intersection. Replacing $f$ by a power of $f$, one can suppose that $V$ is invariant by $f$ and every $z_i$ fixed by $f$.

\begin{lemma}\label{l.disks} One can construct a family of closed disks $(D_i)_{i\in I}$ in $V$ such that:

\begin{enumerate}

\item $f(D_i)\cap D_i=\emptyset$;

\item $D_i=D_j$ if $z_i\sim z_j$;

\item $D_i\cap D_j=f(D_i)\cap D_j=\emptyset$ if $z_i\not\sim z_j$;

\item the branches of $z_i$ meet $D_i$;

\item the branches of $z_i$ do not meet $D_j$ if $z_i\not\sim z_j$.

\end{enumerate}
\end{lemma}
\begin{proof}
 One can find a subset $I'\subset I$ such that for every $i\in I$, there exists a unique $i'\in I'$ such that $z_i\sim z_{i'}$. For every $i'\in I'$, fix a point $z'_{i'}$ in a stable branch of $z_{i'}$, then a closed disk $D_{i'}\subset V$ neighborhood of $z'_{i'}$ and extend the family $(D_{i'})_{i'\in I'}$ to a family $(D_i)_{i\in I}$ that satisfies (2). By Corollary \ref{c.genericconsequences}, $z_i$ does not belong to the closure of the branches of $z_j$ if $z_i\not\sim z_j$. So the branches of $z_i$ do no intersect the closure of the branches of $z_j$. Consequently, if the neighborhoods $D_{i'}$ are chosen sufficiently small, then the five conditions are satisfied.\end{proof}

We construct a family of loops $(\lambda^u_i)_{i\in I}$ in the following way: we consider the first points $x_i^u$ and $y_i^u$ where the unstable branches meet $D_i$, we note $\alpha^u_i$ the simple path included in $W^u(z_i)$ joining $x_i^u$ to $y_i^u$. Of course, it  contains $z_i$. Then, we choose a simple path $\beta^u_i\subset D_i$ joining $y_i^u$ to $x_i^u$ and we define $\lambda^u_i=\alpha^u_i\beta^u_i$ by concatenation.

Similarly, we construct a family of loops $(\lambda^s_i)_{i\in I}$ in the following way: we consider the first points $x_i^s$ and $y_i^s$ where the stable branches meet $f(D_i)$, we note $\alpha^s_i$ the simple path included in $W^s(z_i)$ joining $x_i^s$ to $y_i^s$ . Then, we choose a simple path $\beta^s_i\subset f(D_i)$ joining $y_i^s$ to $x_i^s$ and we define $\lambda^s_i=\alpha^s_i\beta^s_i$.

Exceptionally, until the end of the section, the notation $[\lambda]$ will mean the homology class  in $H_1(\check V,\Z_2)$ and  $\wedge$ will mean the symplectic ``intersection form'' defined on $H_1(\check V,\Z_2)$.

\begin{lemma}\label{l.intersection} We have the following: \begin{enumerate}
\item if $W^u(z_i)\cap W^s(z_i)=\{z_i\}$, then $[\lambda^u_i]\wedge[\lambda^s_i]=1$;
\item if $i\not=j$ and $ W^u(z_i)\cap W^s(z_j)=\emptyset$, then $[\lambda^u_i]\wedge[\lambda^s_j]=0$;
\item if $i\not \sim j$, then $[\lambda^u_i]\wedge[\lambda^s_j]=0$.

\end{enumerate}
\end{lemma}

\begin{proof}
Note that for every $i$, $j$ in $I$, the sets $$f^{-1}(\alpha_i^u)\cap D_j,\enskip f(\alpha_i^s)\cap f(D_j),\enskip f(D_i)\cap D_j,$$ are all empty, as are the sets
$$\alpha_i^u\cap f(D_j),\enskip \alpha_i^s\cap D_j.$$
 Consequently $$\alpha^u_i\cap \beta^s_j=\alpha^s_i\cap \beta^u_j=\beta^s_i\cap \beta^u_j=\emptyset.$$
One deduces that
 $$\lambda^u_i\cap \lambda^s_j=\alpha^u_i\cap\alpha^s_j\subset W^u(z_i)\cap W^s(z_j),$$
which implies that
$$W^u(z_i)\cap W^s(z_i)=\{z_i\}\Rightarrow [\lambda^u_i]\wedge[\lambda^s_j]=1$$ and
$$ W^u(z_i)\cap W^s(z_j)=\emptyset\Rightarrow[\lambda^u_i]\wedge[\lambda^s_j]=0.$$
In particular
$$ z_i\not\sim z_j\Rightarrow W^u(z_i)\cap W^s(z_j)=\emptyset\Rightarrow[\lambda^u_i]\wedge[\lambda^s_j]=0.$$

\end{proof}

Continuing with the proof of Proposition \ref{p.cycle}, $H_1(\check V, \Z^2)$ has dimension $2g'$ since $g(V)=g'$. We conclude that the family $([\lambda^s_i])_{i\in I}$ is linearly dependent. So we can find a non empty  linearly dependent sub-family $([\lambda^s_i])_{i\in J}$ that is minimal for this property. This means that every $[\lambda^s_i]$, $i\in J$, is a linear combination of the $[\lambda^s_j]$, $j\in J\setminus\{i\}$.

\begin{lemma}\label{l.cycle} For every $i\in J$ there exists $j\in J$ such that one of the unstable branches of $z_i$ meets one of the stable branches of $z_j$.

\end{lemma}
\begin{proof} We will use Lemma \ref{l.intersection} twice. Suppose that $W^u(z_i)\cap W^s(z_j)=\emptyset$ for every $j\in J\setminus\{i\}$. Then we have $ [\lambda^u_i]\wedge [\lambda^s_j]=0$ for every $j\in J\setminus\{i\}$.
The class $[\lambda^s_i]$ being a linear combination of the $[\lambda^s_j]$, $j\in J\setminus\{i\}$, it implies that $[\lambda^u_i]\wedge [\lambda^s_i]=0$. We conclude that
$W^u(z_i)\cap W^s(z_i)\not=\{z_i\}$ .\end{proof}

To conclude the proof of Proposition \ref{p.cycle} it remains to say that there is a sub-family $(z_k)_{k\in \Z/r\Z}$ of the family $(z_i)_{i\in I}$ such that for every $k\in\Z/r\Z$, there is an unstable branch of $z_k$ and a stable branch of $z_{k+1}$ that intersect. The common class is homoclinic.\end{proof}

\begin{remark*} Looking at the particular case where $g'=0$, Proposition  \ref{p.cycle} asserts that if $S$ is the $2$-sphere, then if $f\in\cG^r_{\omega}(S)$,   the stable and unstable branches of  every hyperbolic periodic point intersect. In fact the proof above tells us the following: if $f$ is a diffeomorphism of class $C^1$ of the $2$-sphere and $z$ is a saddle hyperbolic fixed point such that the closure of its four branches contain a common point $z'\not\in \mathrm{fix}(f)$, then the stable and unstable manifolds of $z$ intersect. Indeed if $D$ is a closed disk containing $z$ in its interior and satisfying $f(D)\cap D=\emptyset$, one can construct two loops $\lambda^s$ and $\lambda^u$, containing neighborhoods of $z$ in $W^s(z)$ and $W^u(z)$ respectively, and such that $\lambda^s\cap \lambda^u\subset W^s(z)\cap W^u(z)$. The two loops must have point of intersection beside $z$ because $H_1(S^2,\Z_2)=0$. Note that we get Theorem \ref{t.main} in the case of the sphere.
\end{remark*}

\section{Proof of Theorem \ref{t.numberhyp}} \label{s:proof}

Theorem \ref{t.numberhyp} is an immediate consequence of Proposition \ref{p.localLefschetz} and of the following result:

\begin{proposition}\label{p.finitecovering}
Let $f\in \cG^r_{\omega}(S)$ and $\kappa\in\cE(f)$. If $\# \kappa> \max(0,2g(\kappa)-2)$, then $\kappa$ is a homoclinic class. \end{proposition}

\begin{proof} Let $V$ be a periodic regular domain of genus $g'=g(\kappa)$ containing $\kappa$. If $g'=0$, one can apply Proposition \ref{p.cycle}: every point of $\kappa$ has a homoclinic transverse intersection and so $\kappa$ is a homoclinic class. We suppose from now on that $g'>0$, which implies that there exist a least $2g'-1$ points in $\kappa$. We fix $m$ such that
$$m(2g'-1)>2m(g'-1)+2.$$
It is sufficient to take $m\geq 3$. We can find two simple loops $\lambda_1$ and $\lambda_2$ in $V$ with a unique intersection point and real algebraic intersection number equal to $1$.
We denote $S'$ the $m$-sheet covering space of $S$ obtained by cutting and gluing cyclically $m$ copies of $S$ along $\lambda_1$. We denote $\pi: S'\to S$ the covering projection. We know that $S'$ is a closed surface of genus $m(g-1)+1$ and that $V'=\pi^{-1}(V)$ is a  regular domain of genus $m(g'-1)+1$ , because it is connected. Let $z_0$ be a fixed point of $f$ and $z'_0$ be a lift of $z_0$ in $S'$. We denote $\pi_1(S',z'_0)$ the fundamental group of $S'$ with base point $z'_0$ and $\pi_1(S,z_0)$ the fundamental group of $S$ with base point $z_0$. The image $\pi_*(\pi_1(S',z'_0))$ is a subgroup of index $m$ of $\pi_1( S, z_0)$ and $f_*$ acts as a permutation on the set of subgroup of index $m$ of $\pi_1(S,z_0)$. The fundamental group $\pi_1(S,z_0)$ being of finite type, one knows by a theorem of Hall \cite{Hal} that $\pi_1(S,z_0)$ contains finitely many subgroups of index $m$ and consequently, that there exists $q\geq 1$ such that $f_*^q$ fixes $\pi_*(\pi_1(S',z'_0))$. By the lifting theorem, one deduces that $f^q$ can be lifted to a diffeomorphism $f'$ of $S'$ that fixes $z'_0$. Note now that:

\begin{itemize}
\item  $f'$ preserves the lifted form $\pi_*(\omega)$;

\item  $f'$ satisfies the conditions (G1), (G2) and (G3);

\item  $V'$ is periodic;

\item $\pi^{-1}(\kappa)\subset \mathrm{per}_h(f')$;

\item $\#\pi^{-1}(\kappa)=m\#\kappa\geq m(2g'-1)>2m(g'-1)+2$.

\end{itemize}

By Proposition \ref{p.cycle}, there exists a point of $\pi^{-1}(\kappa)$ with  homoclinic intersection. It projects onto a point of $\kappa$ with the same property. So $\kappa$ is a homoclinic class.\end{proof}

\begin{remark*}
Observe that Theorem \ref{t.numberhyp} implies Theorem \ref{t.main} in the case that there is a pseudo Anosov component in the Nielsen-Thurston decompositon of $f$.
\end{remark*}

Recall that if $f\in \cG^r_{\omega}(S)$, then $\#\mathrm{per}_h(f)=2g-2$ if and only if $\#\mathrm{per}(f)=2g-2$. Let us conclude this section by explaining the dynamical properties of a diffeomorphism $f\in \cG^r_{\omega}(S)$ such that $\#\mathrm{per}_h(f)=2g-2$.

\begin{proposition}\label{p.dynamicalproperties}  Let $f\in \cG^r_{\omega}(S)$ such that $\#\mathrm{per}_h(f)=2g-2$. Then:
\begin{enumerate}
\item there is no periodic regular domain but $S$;
\item every periodic continuum is a point or $S$;
\item every invariant  open set is connected, has genus $2g$ and its complement is contained in an open disk $D$;
\item $f$ is transitive;
\item there is a unique class in $\cE(f)$;
\item every branch is dense;
\item a stable branch and an unstable branch do not intersect.
\end{enumerate}
\end{proposition}

\begin{proof} The assertions (1) and (5) are immediate consequences of Proposition \ref{p.localLefschetz} and Proposition \ref{p.finitecovering} and the assertion (2) an immediate consequence of (1). The assertion (4) follows from (2). Indeed, to prove that $f$ is transitive, one needs to prove that every invariant open set is dense. Fix a connected component. It is periodic, so its closure is a continuum not reduced to a point. By (2), it is equal to $S$. The assertion (6) also follows from (2) because the closure of a branch is a continuum not reduced to a point.

Now, let us prove (3). If $V$ is an invariant open set,  its connected components are periodic and not reduced to  a point. They are dense, and so $V$ is connected. If $\lambda$ is a simple loop non homologous to zero, and if $S'$ is the $2$-fold covering space obtained by pasting two copies of $S$ along $\lambda$, then the lift of $V$ is connected. Indeed, as explained in the proof of Proposition \ref{p.finitecovering}, there exists $q\geq 1$ such that $f^q$ can be lifted to a diffeomorphism $f'\in \cG^r_{\pi^*(\omega')}(S')$, where $\pi: S'\to S$ is the covering projection. This lift has exactly $4g-4$ hyperbolic periodic points and so $\pi^{-1}(V)$ is connected. This implies that there exists a loop $\lambda'\subset V$ such that $[\lambda']\wedge [\lambda]>0$.   Thus, one can find a compact connected surface with boundary $\Sigma\subset V$ such that the image of the morphism $(i')_*: H_1(\Sigma,\R)\to H_1(S,\R)$ induced by the inclusion $i': \Sigma\hookrightarrow S$  is equal to the image of the morphism $i_*: H_1(V,\R)\to H_1(S,\R)$ induced by the inclusion $i: V\hookrightarrow S$. So, for every boundary circle $\lambda$ of $\Sigma$ and every loop $\lambda'\subset V$, one has $[\lambda']\wedge [\lambda]=0$. One deduces that $[\lambda]=0$, so $\lambda$ bounds a subsurface $\Sigma'$ whose interior is disjoint from $\Sigma$. But for every $\lambda'\subset\Sigma'$ and every $\lambda\subset V$, one has $[\lambda']\wedge [\lambda]=0$, and so $[\lambda']=0$. One deduces that $\Sigma'$ is a closed disk. Consequently, the complement of $V$ can be included in a finite union of closed disks. Then, it is easy to construct a unique disk that contains this complement.

It remains to prove (7). Suppose that $z$ and $z'$ are hyperbolic periodic points such that one of the stable branches of $z$, denoted $\Gamma^u$,  meets one of the unstable branches of $z'$, denoted $\Gamma^s$.
By property (G2) the branches have a transverse intersection and so it is possible to construct a simple loop $\lambda$, concatenation of a sub-path of $\Gamma^u$ and of a sub-path of $\Gamma^s$. If $\lambda$ is homological to zero, its bounds two subsurfaces.  The denseness of the branches implies that every stable branch meets $\Gamma^u$ and every unstable branch meets $\Gamma^s$ and so there are homoclinic intersections, in contradiction with the hypothesis. If $\lambda$ is not homological to zero, we consider the $2$-fold covering space $S'$ as above. The denseness of the branches of $f'$, implies that every branch of $f'$ intersects one of the two lifts of $\lambda$ and consequently, every branch of $f$ meets $\lambda$. We conclude as in the first case.
 \end{proof}

\section{Dehn twist maps and the intersection property}\label{s.dehn}

The main object of this section is the study of homeomorphisms isotopic to a Dehn twist map. We will suppose that $g\geq 2$ in the whole section. We recall the definition of a Dehn twist map.

\begin{definition}
A {\it Dehn twist map} of $S$ is an orientation preserving homeomorphism $h$ of $S$ that satisfies the following properties:

\begin{itemize}
\item there exists a non empty finite family $(A_{i})_{i\in I}$ of pairwise disjoint invariant essential closed annuli;

\item no connected component of $S\setminus \cup_{i\in I}A_i$ is an annulus;

\item $h$ fixes every point of $S\setminus \cup_{i\in I}A_i$;

\item for every $i\in I$, the map $h_{\vert A_i}$ is conjugate to $\tau^{n_i}$, $n_i\not=0$, where $\tau$ is the homeomorphism of $\T\times[0,1]$ that is lifted to the universal covering space $\R\times[0,1]$ by $\widetilde\tau:(x,y)\mapsto (x+y,y)$.

\end{itemize}


\end{definition}

We will fix from now on a Dehn  twist map $h$ and a homeomorphism $f$ isotopic to $h$. We will denote $(A_{i})_{i\in I}$ the family of twisted annuli and  $(n_i)_{i\in I}$  the family of twist coefficients.

The closure of a connected component of $S\setminus \cup_{i\in I}A_i$  is a surface with non empty boundary. Choose an annulus $A=A_{i_0}$ and write $\lambda$ and $\lambda'$ its boundary circles. The circle $\lambda$ belongs to the closure $\Sigma$ of
a connected component of $S\setminus \cup_{i\in I}A_i$  and similarly the circle $\lambda'$ belongs to the closure $\Sigma'$ of
another connected component. It may happen that $\Sigma=\Sigma'$.  Let $\widetilde A$ be a connected component of $\widetilde\pi^{-1}(A)$, where $\widetilde\pi: \widetilde S\to S$ is the universal covering space. The stabilizer of $\widetilde A$ in the group $G$ of covering automorphisms is a cyclic group generated by an element $T_0\not=\mathrm{Id}$. The boundary of $\widetilde A$ is the union of two lines $\widetilde \lambda$ and $\widetilde\lambda'$ that lift $\lambda$ and $\lambda'$ respectively. There exists a unique lift $\widetilde h$ of $h$ that fixes every point of $\widetilde\lambda$.  This lift coincides on $\widetilde \lambda'$ with $T_0^{n_i}$ or with $T_0^{-n_i}$. Replacing $T_0$ with $T_0^{-1}$  if necessary, one can suppose that the first situation occurs. There exists a unique connected component $\widetilde \Sigma$ of $\widetilde\pi^{-1}(\Sigma)$ that contains $\widetilde \lambda$ and a unique connected component $\widetilde \Sigma'$ of $\widetilde\pi^{-1}(\Sigma')$ that contains $\widetilde \lambda'$. Moreover $\widetilde h$ fixes every point
 of $\widetilde \Sigma$ and coincides with  $T_0^{n_i}$ on $\widetilde \Sigma'$. Note that $\widetilde h$ commutes with $T_0$ and so lifts a homeomorphism $\widehat h$ of the open annulus $\widehat S=\widetilde S/T_0$. The map $\widehat h$ is isotopic to the identity, which means that it preserves the orientation and fixes the two ends of $\widehat S$. There exists an isotopy from $h$ to $f$, uniquely defined up to homotopy (because $g\geq 2$) and this isotopy can be lifted to an isotopy from $\widetilde h$ to a lift $\widetilde f$ of $f$. Here again $\widetilde f$ commutes with $T_0$ and lifts a homeomorphism $\widehat f$ of $\widehat S$ isotopic to the identity.

 One can furnish $S$ with a Riemannian metric of negative curvature, suppose that $\widetilde S$ is nothing but the disk $\D=\{z\in \C\,\vert\, \vert z\vert <1\}$ and that $G$ is a group of M\H obius automorphisms of $\D$. It is well known that $\widetilde h$ and $\widetilde f$ can be extended to homeomorphisms of  $\overline{\D}=\{z\in \C\,\vert\, \vert z\vert \leq1\}$ and that their extensions coincide on  $\S=\{z\in \C\,\vert\, \vert z\vert =1\}$. The extension is defined by the following property: if $\widetilde \gamma$ is a geodesic of $\D$ joining $\alpha\in \S$ to $\omega\in \S$ that lifts a closed geodesic $\gamma$ of $S$, and if $\gamma'$ is the unique closed geodesic freely homotopic to $f(\gamma)$, there exists a unique lift $\widetilde\gamma'$ of $\gamma'$ that remains at bounded distance of $\widetilde f(\widetilde\gamma)$ and this geodesic joins $\widetilde f(\alpha)$ to $\widetilde f(\omega)$.  Let us give an equivalent interpretation. For every $T\in G$, there exists $T'\in G$ such that for every $z\in \widetilde S$, the images by $\widetilde h$ and $\widetilde f$ of a path joining $z$ to $T(z)$ join $\widetilde h(z)$ to $T'(\widetilde h(z))$ and $\widetilde f(z)$ to $T'(\widetilde f(z))$ respectively. Writing $T'=\widetilde h_*(T)=\widetilde f_*(T)$, one gets a morphism $\widetilde h_*=\widetilde f_*: G\to G$. Now, as we suppose that $G$ consist of M\H obius transformations, we know that the $\alpha$-limit set of a non trivial element $T\in G$ is reduced to a point $\alpha(T)$ and the $\omega$-limit set is reduced to a point $\omega(T)$. The extension is defined by the following property: for every $T\in G\setminus\{\mathrm{Id})$, one has
 $\widetilde f(\alpha(T))=\alpha(\widetilde f_*(T))$ and $\widetilde f(\omega(T))=\omega(\widetilde f_*(T))$.

 In the situation above, the points $\alpha(T_0)$ and $\omega(T_0)$ are  distinct and fixed by $\widetilde f$ and the quotient space $\left(\overline \D\setminus \{\alpha(T_0), \omega(T_0)\}\right)/T_0$ is a compact annulus $\overline{\widehat S}$ that compactifies $\widehat S$. The maps $\widehat f$ and $\widehat h$ can be extended to this compact annulus and have the same extension on the boundary circles. The fact that $\Sigma$ is not an annulus implies that one of the  two component of $\S\setminus \{\alpha(T_0), \omega(T_0)\}$ meets the closure of $\widetilde \Sigma$ in $\overline \D$. As $\widetilde \Sigma$ consists of fixed points of $\widetilde h$, we deduce that $\widetilde h$ fixes some points of this component and $\widetilde f$ fixes the same points. Similarly, $\widetilde f\circ T_0^{-n_i}$ fixes some points of the other component.  Consequently, $\widehat f$ admits fixed points on both boundary circles of $\overline{\widehat A}$ and so,  the rotation numbers (as elements of $\T$) of both circles are equal to $0$. But the difference between the real rotation numbers is non zero, it is equal to $n_i$. The map $\widehat f$ satisfies a boundary twist condition.

\begin{proposition} \label{prop:intprop} If $f$ has finitely many periodic points, then there exists an essential simple loop $\widehat \lambda''$ such that $\widehat f(\widehat \lambda'')\cap\widehat \lambda''=\emptyset$.\end{proposition}

\begin{proof} If
$\widehat f$ satisfies the intersection property, which means that there is no essential simple closed loop $\widehat \lambda''$ such that $\widehat f(\widehat \lambda'')\cap\widehat \lambda''=\emptyset$, then by Theorem \ref{t.PBtheorem} we have that for every rational number $p/q\in(0,n_i)$, written in an irreducible way, there exists a periodic point $\widehat z$ of period $q$ and rotation number $p/q$, which means lifted by a point $\widetilde z$ such that $\widetilde f^q(\widetilde z)=T_0^p(\widetilde z)$. There exists a compact set $ \widehat K\subset \widehat S$ that meets every periodic orbit of rotation number $\rho\in[1/3,2/3]$. Indeed, if this is not the case, one can find a sequence of periodic orbits $( \widehat O_n)_{n\geq 0}$ that converges in the Hausdorff topology to a  subset of a boundary circle of $\overline{\widehat S}$,  such that the rotation number of $ \widehat O_n$ belongs to $[1/3,2/3]$. But this would imply that the rotation number of this boundary circle also belongs to $[1/3,2/3]$ and we know that it is equal to an integer. Consequently $ \widehat K$ contains infinitely many periodic points. The set $ \widehat K$ being compact, there exists $r\geq 1$ such that every point $z\in S$ has at most $r$ preimages in $ \widehat K$ by the covering projection $\widehat\pi :\widehat S\to S$. As $\widehat\pi$ sends periodic points of $\widehat f$ onto periodic points of $f$, we can conclude that $f$ has infinitely many periodic points.
\end{proof}

Let us show an example where the intersection property must be verified.

\begin{proposition} \label{prop:zerohom} Suppose that the boundary circles of $A_{i_0}$ are homologous to zero and that $f$ preserves a finite Borel measure $\mu$ with total support. Then $\widehat f$ satisfies the intersection property and so $f$ has infinitely many periodic points.
\end{proposition}

\begin{proof} We will argue by contradiction. If $\widehat f$ does not satisfy the intersection property, there exists an essential simple loop $\widehat\lambda''$ of $\widehat S$ such that $\widehat f(\widehat\lambda'')\cap \widehat \lambda''=\emptyset$. Of course $\widehat\lambda''$ can be perturbed in such a way that it projects onto a loop $\lambda''\subset S$ satisfying $\mu(\lambda'')=0$. The loop $\widehat \lambda''$ can be lifted in a line $\widetilde\lambda''$ of $\widetilde S$ such that $\widetilde f(\widetilde\lambda'')\cap \widetilde\lambda''=\emptyset$. Let us orientate $\widetilde\lambda''$ in such a way that $\widetilde f(\widetilde\lambda'')$ is on the left of $\widetilde\lambda''$.

This implies that $\ \overline{L(\widetilde \lambda'')}\subset \widetilde f^{-1}(L(\widetilde \lambda''))$ and consequently that $ \overline{L(T(\widetilde \lambda''))}\subset \widetilde f^{-1}(L(\widetilde f_*(T)(\widetilde \lambda'')))$ for every $T\in G$.

One can define a {\it dual function} $\delta_{\lambda''}$ on the complement of $\lambda''$ because $\lambda''$ is homologous to zero, being freely homotopic to the boundary circles of $A_{i_0}$. Indeed, since $\lambda''$ is homologous to zero, the intersection number of a path with $\lambda''$ (once is well defined) depends only on the end points of the path; hence one can define a function $\delta_{\lambda''}: S\setminus\{\lambda''\}\to \Z$ such that $\delta_{\lambda''}(z_2)-\delta_{\lambda''}(z_1)$ is equal to the intersection number of a path joining $z_1$ to $z_2$ with $\lambda''.$ This dual function is well defined up to an additive constant.

Let us proof that $\delta_{\lambda''}(f(z))\leq \delta_{\lambda''}(z)$ if $z\not\in \lambda''\cup f^{-1}(\lambda'')$ and that there exists a non empty open set $U\subset S\setminus  (\lambda''\cup f^{-1}(\lambda''))$ such that $\delta_{\lambda''}(f(z))<\delta_{\lambda''}(z)$ if $z\in U$.
Fix $z\not\in \lambda''\cup f^{-1}(\lambda'')$ and choose a lift $\widetilde z\in \widetilde S$ of $z$. We have the following formula:
$$\begin{aligned}\delta_{\lambda''}(f(z))-\delta_{\lambda''}(z)= &\# \{T\in G\,\vert \,\widetilde f(\widetilde z)\in R(T(\widetilde \lambda''))\enskip \mathrm{and} \enskip \widetilde z\in L(T(\widetilde \lambda''))\}\\
-&\# \{T\in G\,\vert \,\widetilde f(\widetilde z)\in L(T(\widetilde \lambda''))\enskip \mathrm{and} \enskip \widetilde z\in R(T(\widetilde \lambda''))\}.
\end{aligned}$$
We can write the formula in a different way. Denote $G/\langle T_0\rangle$ the set of left cosets, where  $\langle T_0\rangle$ is the cyclic group generated by $T_0$, and keep the notation $\widetilde f_*$ for the map $\tau=T\langle T_0\rangle\mapsto \widetilde f_*(\tau)=\widetilde f_*(T)\langle T_0\rangle$. If $\tau =T\langle T_0\rangle$, write $\tau(\widetilde \lambda'')=T(\widetilde \lambda'')$. If $\widetilde \lambda$ is an oriented line of $\widetilde S$, write $\delta_{\widetilde \lambda}$ for the map equal to $1$ on $R(\widetilde \lambda)$ and $0$ on $L(\widetilde \lambda)$.  Noting that the sums below have finitely many non zero terms, we have
$$\begin{aligned}&\enskip\enskip\enskip\enskip\enskip\delta_{ \lambda''}(f(z))-\delta_{\lambda''}(z)=\\ &= \sum_{\tau\in G/\langle T_0\rangle} \delta_{\tau (\widetilde \lambda'')} (\widetilde f(\widetilde z))-\delta_{\tau (\widetilde \lambda'')} (\widetilde z)\\
 &= \sum_{\tau\in G/\langle T_0\rangle} \delta_{\widetilde f^{-1}(\tau (\widetilde \lambda''))} (\widetilde z)-\delta_{\tau (\widetilde \lambda'')} (\widetilde z)\\
 &=  \sum_{\tau\in G/\langle T_0\rangle} \delta_{\widetilde f^{-1}(\tau (\widetilde \lambda''))} (\widetilde z)-\delta_{\widetilde f_*^{-1}(\tau) (\widetilde \lambda'')} (\widetilde z)+\sum_{\tau\in G/\langle T_0\rangle} \delta_{\widetilde f_*^{-1}(\tau) (\widetilde \lambda'')} (\widetilde z)-\delta_{\tau (\widetilde \lambda'')} (\widetilde z)\\
  &=  \sum_{\tau\in G/\langle T_0\rangle} \delta_{\widetilde f^{-1}(\tau (\widetilde \lambda''))} (\widetilde z)-\delta_{\widetilde f_*^{-1}(\tau) (\widetilde \lambda'')} (\widetilde z).
\end{aligned}$$
Each function $\delta_{\widetilde f^{-1}(\tau (\widetilde\lambda''))} -\delta_{\widetilde f_*^{-1}(\tau) (\widetilde\lambda'')} $ is non positive and takes a negative value on a non empty open set because $\overline{L(\widetilde f_*^{-1}(\tau) (\widetilde \lambda''))}\subset \widetilde f^{-1}( L((\tau (\widetilde \lambda'')))$. We deduce that  $\delta_{\lambda''} \circ f -\delta_{\lambda''} $  is a non positive function defined $\mu$ everywhere and that takes a negative value on a set of positive measure. This contradicts the fact that $f$ preserves $\mu$.
\end{proof}

\begin{remark}
 Let us give another proof that consists in showing that $\lambda''$ is a simple loop disjoint from its image by $f$. This loop will bound two surfaces, an attracting one and a repelling one, which of course is impossible because $f$ preserves $\mu$. The given orientation on $\lambda''$ induces naturally an orientation on the boundary circle $\lambda$ of $A_{i_0}$. Denote $\Xi$ and $\Xi'$ the subsurfaces bounded by $\lambda$ located on its left and on its right respectively. We denote $\widetilde\Xi$ the connected component of $\widetilde \pi^{-1}(\Xi)$, that contains $\widetilde\lambda$. It is a surface whose boundary is a union of images of $\widetilde \lambda$ by covering automorphisms. Moreover $\widetilde\Xi$ is located on the left side of every boundary lines. In particular, the stabilizer $H$ of $\widetilde\Xi$ in $G$ acts transitively on the set of boundary lines. Note also that $H$ is the group of covering automorphisms of $\Xi$. Let us prove that $T(\widetilde \lambda'')\cap \widetilde\lambda''=\emptyset$ for every $T\in H\setminus \langle T_0\rangle$. We can index the boundary lines of   $\widetilde\Xi$ by classes $\tau\in H/\langle T_0 \rangle$. Note that if $\tau\in H/\langle T_0 \rangle$ and $\tau'\in H/\langle T_0 \rangle$ are distinct, then $\overline{R(\tau(\widetilde \lambda''))}\cap \overline{R(\tau'(\widetilde \lambda''))}$ is compact. Every line $\tau(\widetilde \lambda'')$, $\tau\not=\langle T_0 \rangle$, projects in $\widehat S$ onto a line of $\widehat S$ joining the end of $\widehat S$ on the left of $\widehat \lambda''$ to the same end, we will write it $\tau(\widehat \lambda'')$.  Note that there are finitely many lines $\tau(\widehat \lambda'')$ that intersect $\widehat \lambda''$.

 Consider the set
 $$\widehat K=\overline {R(\widehat \lambda'')}\cap\left(\bigcup_{\tau \in (H/\langle T_0\rangle)\setminus\langle T_0\rangle} \overline{R(\tau (\widehat \lambda''))}\right).$$
 It is compact  and we have $\widehat f^{-1}(\widehat K) \subset \mathrm{int} (\widehat K)$. Indeed, if $\widehat z\in  \overline {R(\widehat \lambda'')}\cap\overline{R(\tau (\widehat \lambda''))}$, then $\widehat f^{-1}(\widehat z)\in  R(\widehat \lambda'')\cap R(\widetilde f_*^{-1}(\tau) (\widehat \lambda''))$. We deduce that $\mathrm{int} (\widehat K)\setminus \widehat f^{-1}(\widehat K) $ is a wandering open set of $\widehat f$, whose backward orbit is contained in $\widehat K$.
 The fact  that  $\widehat f$ preserves a locally finite measure with total support implies that $\widehat f^{-1}(\widehat K) = \mathrm{int} (\widehat K)$.  This implies that $\widehat K$ is open and compact and so is empty. This clearly implies that no line  $\tau(\widehat \lambda'')$,  $\tau \in H/\langle T_0\rangle\setminus\langle T_0\rangle$, meets $\widehat \lambda''$.  Denoting $\widetilde\Xi'$ the connected component of $\widetilde\pi^{-1}(\Xi')$ that contains $\widetilde\lambda$ and $H'$ the stabilizer of $\widetilde\Xi'$ in $ G$, we prove similarly that no line $\tau(\widetilde \lambda'')$, $\tau \in H'/\langle T_0\rangle\setminus\langle T_0\rangle$, meets $\widetilde \lambda''$. We have a tiling of $\widetilde S$ by images of $\widetilde{\Xi}$ or $\widetilde{\Xi'}$  by covering automorphisms and it is easy to deduce that no line $\tau(\widetilde \lambda'')$, $\tau \in G/\langle T_0\rangle\setminus\langle T_0\rangle$, meets $\widetilde \lambda''$. It is also easy to see that $\widetilde f(\widetilde \lambda'')$ is included in the interior of $\widetilde \Xi$ and so does not meet any line $\tau(\widetilde \lambda'')$. This means that $f(\lambda'')\cap\lambda''=\emptyset.$

\medskip
The interest of this second proof is that it permits to deal with the case  where the boundary circles of $A_{i_0}$ are not homologous to zero. Denote $\check S$ the cyclic cover of $S$ associated to $\lambda$, which means the surface $\widetilde S / H$, where $H$ is the normal subgroup defined has follows:
$T\in H$ if for every $\widetilde z$, the path joining $\widetilde z$ to $T(\widetilde z)$ projects onto a loop $\lambda''$ such that $\lambda\wedge\lambda''=0$. The covering $\widehat S$ is an intermediate covering between $\widetilde S$ and $\check S$. Denote $\check \lambda$ the image of $\widehat\lambda$ by the covering projection and $\check f$ the projected map. Of course we cannot conclude that $\widehat f$ has the intersection property because $\lambda$ and $\lambda'$ do not bound a surface, but we can conclude that if $\widehat\lambda''\in \widehat S$ is a simple loop disjoint from its image by $\widehat f $, and $f$ preserves a finite measure with total support, then $\widehat \lambda''$ projects in $\check S$ onto a simple loop $\check \lambda''$ homotopic to $\check\lambda$ and disjoint from its image by $\check f$. Indeed $\check \lambda$ bounds two subsurfaces and the second proof of  Proposition \ref{prop:zerohom} can be transcribed word to word. We can state this as a proposition.
\end{remark}

\begin{proposition} \label{prop:nonzerohom} Suppose that the boundary circles $\lambda$ and $\lambda'$ of $A_{i_0}$ are not homologous to zero and that $f$ preserves a finite measure with total support $\mu$. If $\widehat f$ does not satisfy the intersection property, then there exists a simple loop  $\check\lambda''$  in the cylic covering space $\check S$ associated to $\lambda$ that project in $S$ onto a loop homotopic to $\lambda$ and that is disjoint by the lift $\check f$ of $f$ that is lifted by $\widehat f$.
\end{proposition}

\section{Proof of Theorem \ref{t.numberminhyp}}\label{s.numberminhyp}

\bigskip
Let us state now the key result of the section

\begin{proposition} \label{prop:nonzerohomperiodic} If $f\in\mathcal{G}^r_{\omega}(S)$ is isotopic to a Dehn twist map, then $f$ has infinitely many periodic points.
\end{proposition}

\begin{proof} We will argue by contradiction by supposing that there are finitely many periodic points. Taking a power of $f$ if necessary we can suppose that there exist exactly $2g-2$ fixed saddle points, with fixed stable and unstable branches,  and no other periodic point. By Proposition \ref{p.dynamicalproperties}, the branches are dense and do not intersect. By Proposition \ref {prop:zerohom}, the boundary circles of the twisted annuli are not homologous to zero.  Choose an annulus  $A_{i_0}$ in the family $(A_i)_{i\in I}$ of twisted annuli. There is no loss of generality by supposing that the twist coefficient $n_{i_0}$ is positive.  Let $\check\lambda''$ be a loop  associated to  $A_{i_0}$, defined by Proposition \ref{prop:nonzerohom}. We keep the notations introduced in Section  \ref{s.dehn} and denote $\check T$ a generator of the group of covering automorphisms of the covering map $\check \pi: \check S\to S$. If $r\geq 1$ is large enough, then $\check T^r(\check \lambda'')\cap \check  \lambda''=\emptyset$. Replacing $\check T$ with its inverse, we can always suppose that $\check T^r(\check\lambda'')$ is on the left of $\check \lambda''$.

\begin{lemma} \label{lemma:liftedtofixed} The fixed points of $f$ are all lifted to fixed points of $\check f$.
\end{lemma}
\begin{proof}
If $\check z$ lifts a fixed point $z$ of $f$, there exists $s\in\Z$ such that $\check f(\check z)=\check T^s(\check z)$ and so $\check f^n(\check z)=\check T^{ns}(\check z)$ for every $n\geq 1$. Consequently, to prove that $\check z$ is fixed, it is sufficient to prove that it is periodic. For every integers $m<n$, denote $\check S_{[m,n]}$ the compact surface bordered by $\check T^{mr}(\check \lambda'')$ and $\check T^{nr}(\check \lambda'')$. The open surface $\bigcup_{k\in\Z} \check f^k(\check S_{[0,1]})$ is invariant and contains two ends, one sink and one source. The sink admits an attracting annular neighborhood on the left of $\check T^r(\check\lambda'')$ and the source a repelling annular neighborhood on the right of $\check\lambda''$. The end compactification of  $\bigcup_{k\in\Z} \check f^k(\check S_{[0,1]})$ is a closed surface of genus  $r(g-1)$and the extended map has a sink and a source,  both of them having a Lefschetz index equal to one for all the iterates of the extended map. Applying what has been done in Section  \ref{s:minoration} to the extended map on this end compactification, we deduce that there exist at least $r(2g-2)$ periodic points beside the two ends, all of them in the interior of $ \check S_{[0,1]}$. But there exists at most $r(2g-2)$ lifts of fixed points of $f$ in the interior of $ \check S_{[0,1]}$. Consequently there exists exactly $r(2g-2)$ lifts of fixed points of $f$ in the interior of $ \check S_{[0,1]}$, all of them periodic, and so, all of them fixed. We immediately deduce the conclusion of the lemma.\end{proof}

\begin{lemma} \label{lemma:liftedtofixed2} The unstable branches of a fixed point of $\check f$ inside $\check S_{[m,n]}$ meet every $\check T^{kr}(\check \lambda'')$, $k\geq n$, and do not meet any $T^{kr}(\check \lambda'')$, $k\leq m$.  Its stable branches meet every $\check T^{kr}(\check \lambda'')$, $k\leq m$, and do not meet any $\check T^{kr}(\check \lambda'')$, $k\geq n$.
\end{lemma}
\begin{proof} It is sufficient to prove the first sentence, the proof of the other one being similar. The fact that the unstable branches of a fixed point of $\check f$ inside $\check S_{[m,n]}$ do not meet any $\check T^{kr}(\check \lambda'')$, $k\leq m$, is an immediate consequence of the fact that every loop $\check T^{k}(\check \lambda'')$, is sent on its left by $\check f$.  Fix $k\geq n$ and consider the closed surface $\dot S= \check S/\check T^{(k+1-m)r}$. The map $\check f$ lifts a map $\dot f$, which itself is a lift of $f$, and $\dot f$ belongs to ${\mathcal G}^r_{\dot \pi ^*(\omega)}(\dot S)$, where $\dot\pi: \dot S\to S$ is the covering projection. The map $\dot f$ having finitely many periodic points, we know that the branches of its periodic points are dense. Consequently, every branch of a fixed point of $\check f$ inside $\check S_{[m,n]}$ must intersect a manifold $\check S_{[k+p(k+1-m),\,k+1+p(k+1-m)]}$, where $p\in\Z$.  If this branch is unstable it cannot meet any $\check S_{[k+p((k+1-m),\,k+1+p(k+1-m)]}$, $p<0$, and so it meets a manifold  $\check S_{[k+p(k+1-m),\,k+1+p(k+1-m)]}$, where $p\geq 0$. Consequently, it intersects $\check T^{kr}(\check \lambda'')$.
\end{proof}

Choose a fixed point $\check z$ of $\check f$ in $\check S_{[-1,0]}$ and an unstable branch of $\check z$. Consider the first point $\check x$ where the branch meet $\check \lambda''$ and the first point $\check y$ where the other unstable branch meet $\check T^r(\check \lambda'')$. Note $\check \alpha^u$ the sub-path of $W^u(\check z)$ that joins $\check x$ to $\check y$. Fix a  lift $\widetilde \lambda''$ of $\check\lambda''$ to the universal covering space and denote $\widetilde S_{[-1,0]}$ the connected component of the preimage of  $\check S_{[-1,0]}$ by the universal covering map, that contains $\widetilde \lambda''$ in its boundary. Choose a lift $\widetilde z\in \widetilde S_{[-1,0]}$ of $\check z$ and denote  $\widetilde\alpha^u$ the lift of $\check\alpha^u$ that contains $\widetilde z$. It joins a lift $\widetilde x$ of $\check x$ to a lift $\widetilde y$ of $\check y$. The algebraic intersection number between $\check \lambda''$ and the unstable path joining $\check z$ to $\check y$ being non zero, there exists a unique lift of $\check\lambda''$ that belongs to the boundary of $\widetilde S_{[-1,0]}$ and that separates $\widetilde z$ and $\widetilde y$. Replacing $\widetilde z$ with another lift if necessary, one can suppose that  $\check \lambda''$ itself separates $\widetilde z$ and $\widetilde y$. The point $\widetilde x$ belongs to a lift of $\check \lambda''$ lying on the boundary of $\widetilde S_{[-1,0]}$. This lift must be different from $\widetilde \lambda''$. Otherwise, at least one of the stable branches of $\widetilde z$ should intersect  $\widetilde \lambda''$, which  is impossible because the stable branches of $\check z$ do not meet $\check\lambda''$.

Similarly,  choose a fixed point $\check z'$ of $\check f$ in $\check S_{[0,1]}$ that projects in $S$ on the same fixed point $z$ of $f$ than $\check z$. Then note $\widetilde S_{[0,1]}$ the connected component of the preimage of  $\check S_{[0,1]}$ by the universal covering map, that contains $\widetilde \lambda''$ in its boundary.  One can find a lift  $\widetilde z'\in \widetilde S_{[0,1]}$ of $\check z'$ and a subpath $\widetilde\alpha^s$ of the stable manifold of $\widetilde z'$ that contains $\widetilde z'$ and that joins a point $\widetilde x'$ belonging to a lift of $\check \lambda''$ lying on the boundary of $\widetilde S_{[0,1]}$ but different from  $\check \lambda''$ to a point $\widetilde y'$ belonging to a lift of $\check T^{-r}(\check\lambda'')$ separated from $\widetilde z'$ by $\widetilde \lambda''$.

The boundary of $\widetilde S_{[-1,0]}$ is a disjoint union of images of $\widetilde \lambda''$ by covering automorphisms, the ones that bound $\widetilde S_{[-1,0]}$ on their right side being the lifts of $\check \lambda''$, the ones that bound $\widetilde S_{[-1,0]}$ on their left side being the lifts of $\check T^{-r}(\check \lambda'')$. There is a natural order on the set of lifts of $\widetilde \lambda''$ bounding $\widetilde S_{[-1,0]}$ and different from $\widetilde \lambda''$: say that {\it $T'(\widetilde \lambda'')$ is below $T''(\widetilde \lambda'')$  relative to  $\widetilde \lambda''$} if there exist two disjoint paths $\widetilde \alpha'$ and $\widetilde \alpha''$ in  $\widetilde S_{[-1,0]}$ , the first one joining a point of $T'(\widetilde \lambda'')$ to a point $\widetilde z\in\widetilde \lambda''$, the second one joining a point of $T''(\widetilde \lambda'')$ to  $T_0(\widetilde z)$. Similarly $\widetilde S_{[0,1]}$ is a disjoint union of images of $\widetilde \lambda''$ by covering automorphisms, the ones that bound  $\widetilde S_{[0,1]}$ on their left side being the lifts of $\check \lambda''$, the ones that bound  $\widetilde S_{[0,1]}$ on their right side being the lifts of $\check T^{r}(\check \lambda'')$. Here again, there is a natural order on the sets of lifts different from $\widetilde \lambda''$. Write $T_1(\widetilde \lambda'')$ the lift of $\check\lambda''$ that contains $\widetilde x$ and $T_2(\widetilde \lambda'')$ the lift of $\check T^{r}(\check\lambda'')$ that contains $\widetilde y$.  Write $T'_1(\widetilde \lambda'')$ the lift of $\check T^{-r}(\check\lambda'')$ that contains $\widetilde y'$ and $T'_2(\widetilde \lambda'')$ the lift of $\check\lambda''$ that contains $\widetilde x'$. Replacing $\widetilde z$ with $T_0^{-s}(\widetilde z)$, $s$ large, if necessary, one can always suppose that the line $T_0(T_1(\widetilde \lambda''))$ is below $T'_1(\widetilde \lambda'')$ relative to $\widetilde \lambda''$.  The choice of the lift $\widetilde f$, which fixes some points at the right of $\widetilde\lambda''$ on the boundary circle of $\widetilde S$, implies that for every $n\geq 0$, the line $f_*^n(T_1)(\widetilde \lambda'')$ is below $T_0(T_1(\widetilde \lambda''))$ relative to $\widetilde \lambda''$ and so below  $T'_1(\widetilde \lambda'')$ relative to $\widetilde \lambda''$. But if $n$ is large enough, then $f_*^n(T_2)(\widetilde \lambda'')$ is above $T'_2(\widetilde \lambda'')$ relative to $\widetilde \lambda''$ because $\widetilde f$ coincides with $T_0^{n_{i_0}}$ on some points at the left of $\widetilde\lambda''$ on the boundary circle. The ends $\widetilde f^n(\widetilde x)$ and $\widetilde f^n(\widetilde y)$ of $\widetilde f^n(\widetilde\alpha ^u)$ belong respectively to $L(f_*^n(T_1)(\widetilde \lambda''))$ and $L(f_*^n(T_2)(\widetilde \lambda''))$. Moreover $\widetilde f^n(\widetilde\alpha ^u)$ does not intersect $\widetilde\alpha ^s$ because $z$ has no homoclinic intersection. Finally, $\widetilde f^n(\widetilde\alpha ^u)$ does not intersect neither $R(T'_1(\widetilde\lambda''))$ nor $R(T'_2(\widetilde\lambda''))$ because $\widetilde f^{-n}(R(T'_1(\widetilde\lambda'')))\subset R(f_*^{-n}(T'_1)(\widetilde\lambda''))$ 	and $\widetilde f^{-n}(R(T'_2(\widetilde\lambda'')))\subset R(f_*^{-n}(T'_2)(\widetilde\lambda''))$.We have found a contradiction because the sets $L(f_*^n(T_1)(\widetilde \lambda''))$ and  $L(f_*^n(T_2)(\widetilde \lambda''))$ are separated by $R(T'_1(\widetilde\lambda''))\cup R(T'_2(\widetilde\lambda''))\cup\widetilde\alpha^s$.

\end{proof}We are now ready to prove Theorem  \ref{t.numberminhyp}.

\begin{proof}[Proof of  Theorem \ref{t.numberminhyp}] The case where $g\leq 1$ has been done in the introduction and so we will suppose that $g\geq 2$.  Let  $f\in \mathcal{G}^r_{\omega}(S)$ be such that $\# \mathrm{per}_h(f)=\max(0, 2g-2)$.  By Thurston-Nielsen decomposition theorem (see Section \ref{ss.NT}), one gets that one of the following situations occurs:
\begin{enumerate}
\item there exists a pseudo-Anosov component in the decomposition of $f$;

\item  there is a power of $f$ that is isotopic to a Dehn twist map;

\item there is a power of $f$ that is isotopic to the identity.

\end{enumerate}

The existence of a  pseudo-Anosov component implies the existence of infinitely many periodic points. Moreover we have seen in this section that the same conclusion occurs in the second situation. So, the only possibility is the third one: there is a power of $f$ that is isotopic to the identity. \end{proof}

\section{Dynamics of maps $f\in{\mathcal G}^r_{\omega}(S)$ isotopic to identity and such that $\#\mathrm{per}_h (f)=2g-2$.} \label{s:dynamics-isotopic-identity}

We will suppose in this section that $S$ has genus greater than $1$ and we fix $f\in{\mathcal G}^r_\omega$ isotopic to the identity such that $\#\mathrm{per}_h(f)=2g-2.$
We denote $\widetilde f$ the natural lift to the covering space. A branch of a hyperbolic periodic point of $f$, respectively $\widetilde f$, will be called a branch of $f$, respectively $\widetilde f$.  Recall what is known about the dynamics of $f$ (see Lemma \ref{l.unlinkedfixedpoints}, Proposition  \ref{prop:Brouwerlines} and Proposition  \ref{p.dynamicalproperties}):

\begin{itemize}
\item there is no periodic point beside the hyperbolic fixed points;
\item $f$ is isotopic to the identity relative to its fixed point set;
\item the branches of $f$ are fixed and dense;
\item the branches do not intersect;
\item $f$ is transitive;
\item there exists an oriented loop $\phi$ non homologous to zero that is lifted to a Brouwer line (possibly singular) of $\widetilde f$.
\end{itemize}

Note that  $[\phi]\wedge \mathrm{rot}_f(\mu)\geq 0$  for every $f$-invariant probability measure $\mu$ and that  $[\phi]\wedge \mathrm{rot}_f(\mu_{\omega})>0$.  Indeed, if $\widehat S$ is the cyclic covering associated to $\phi$, one can lift the isotopy $I$ to an identity isotopy $\widehat I=(\widehat f_t)_{t\in[0,1]}$ on $\widehat S$ joining $\mathrm{Id}$ to a lift $\widehat f$ of $f$. The loop $\phi $ can be lifted to a loop $\widehat\phi$ that bounds a surface $\widehat \Sigma$ on its left such that $f(\widehat \Sigma)\subset \mathrm{int}(\widehat\Sigma)\cup (\mathrm{fix}(\widehat f)\cap\widehat\phi)$ and we have
$$[\phi]\wedge \mathrm{rot}_f(\mu)=\widehat \mu \left(\widehat \Sigma\setminus\widehat f(\widehat\Sigma)\right),$$
where $\widehat \mu$ is the measure on $\widehat S$ that lifts $\mu$.

The goal of this section is to prove the following:
\begin{proposition} \label{prop:nonhomologicalintersection} If $\mu$ and $\mu'$ are two $f$-invariant ergodic probability measures such that  $[\phi]\wedge \mathrm{rot}_f(\mu)>0$  and  $[\phi]\wedge \mathrm{rot}_f(\mu')> 0$, then $\mathrm{rot}_f(\mu)\wedge \mathrm{rot}_f(\mu')= 0$.   \end{proposition}

 Every lift $\widetilde \phi$ of $\phi$ being a line, the orientation of $\phi$ induces an orientation on $\widetilde \phi$ and so defines a natural order $\leq_{\widetilde \phi}$ on $\widetilde \phi $.  We will denote $T_{\widetilde \phi}$ the generator of the stabilizer of  $\widetilde \phi$ in the group $G$ of covering automorphisms induced by the orientation, meaning that $\widetilde z<_{\widetilde \phi}T_{\widetilde \phi} (\widetilde z)$ for every $\widetilde z\in \widetilde\phi$.  We say that a subset of $\widetilde\phi$ {\it has length $\leq m$} if it is contained in the interval of $\widetilde\phi$ joining a point $\widetilde z\in \widetilde\phi$ to $T_{\widetilde \phi}{}^m(\widetilde z)$. The genus of $S$ being larger than $1$, one knows that if $\widetilde \phi$ and $\widetilde \phi'$ are two distinct lifts of $\phi$, then $T_{\widetilde \phi}\not=T_{\widetilde \phi'}$. We have a natural order $\preceq$ on the set of lifts of $\phi$, setting $$\widetilde\phi\preceq \widetilde\phi' \Leftrightarrow R(\widetilde \phi)\subset R(\widetilde \phi')\Leftrightarrow L(\widetilde \phi')\subset L(\widetilde \phi).$$
Finally, say that two different lifts  $\widetilde \phi$ and $\widetilde \phi'$ are {\it adjacent} if they belong to the boundary of the same connected component of $\widetilde\pi^{-1}(S\setminus\phi)$, where $\widetilde \pi: \widetilde S\to S$ is the covering projection.

\medskip

Every lift $\widetilde \phi$ of $\phi$ being a (possibly singular) Brouwer line, the forward orbit of $\widetilde z\in L(\widetilde \phi)$ stays in $L(\widetilde \phi)$ and the backward orbit of $\widetilde z\in R(\widetilde \phi)$ stays in $R(\widetilde \phi)$. Say that the orbit of $\widetilde z\in\widetilde S$ {\it crosses $\widetilde \phi$ at time $k\in\Z$} if $\widetilde f^{k-1}(\widetilde z)\in R(\widetilde \phi)$ and $\widetilde f^k(\widetilde z)\in \overline{L(\widetilde \phi)}$.
The set of lifts of $\phi$ crossed by the orbit of $\widetilde z$ can be indexed by an interval (possibly empty) $J$ of $\Z$, defining an increasing sequence $(\widetilde \phi_k)_{k\in J}$, where two consecutive lifts are adjacent. This sequence will be called the {\it $\phi$-trajectory of $\widetilde z$}. Note that the sequence of crossing times is non decreasing.
By definition we will say that a sequence $(\widetilde \phi_k)_{k\in J}$, is {\it admissible} if there exists $\widetilde z\in\widetilde S$ such that $(\widetilde \phi_k)_{k\in J}$ is a sub-sequence of the $\phi$-trajectory of $\widetilde z$. Note that if this sequence is finite, there exists a neighborhood $U$  of $\widetilde z$ such that $(\widetilde \phi_k)_{k\in J}$ is a sub-sequence of the $\phi$-trajectory of $\widetilde z'$, if $\widetilde z'\in U$. Observe that in the case where $\widetilde z$ projects onto a bi-recurrent point $z$ of $f$, its $\phi$-trajectory is either empty or bi-infinite ($J$ coincides with $\Z$). The last case appears for example if the orbit of $z$ is dense. In this case, the $\phi$-trajectory of $\widetilde z$ contains as a subsequence, the image by a covering automorphism of any given finite admissible sequence. As a consequence, we know that every admissible sequence $(\widetilde \phi_k)_{1\leq k\leq n}$ can be extended to an admissible sequence  $(\widetilde \phi_k)_{0\leq k\leq n+1}$ (we will use this fact later). For a similar reason, if $\widetilde\Gamma$ is a branch of $\widetilde f$  and $(\widetilde \phi_k)_{k\in J}$ is an admissible sequence, there exists $T\in G$  and $\widetilde z\in T(\widetilde\Gamma)$ such that $(\widetilde \phi_k)_{k\in J}$ is a sub-sequence of the $\phi$-trajectory of  $\widetilde z$.

\medskip
Let us define a last notion. Say that there is a {\it crossing} between two admissible paths $(\widetilde \phi_k)_{k\in J}$ and $(\widetilde \phi'_{k'})_{k'\in J'}$ if there exist $k_1<k_2$ in $J$ and $k'_1<k'_2\in J'$ such that:
 \begin{itemize}
 \item $\widetilde \phi_{k_1}\subset L(\widetilde\phi'_{k'_1})$ and $\widetilde \phi'_{k'_1}\subset L(\widetilde\phi_{k_1})$;
  \item $\widetilde \phi_{k_2}\subset R(\widetilde\phi'_{k'_2})$ and $\widetilde \phi'_{k'_2}\subset R(\widetilde\phi_{k_2})$;
    \item if $\gamma$ and $\gamma'$ are paths in the closure of $L(\widetilde\phi_{k_1})\cap L(\widetilde\phi'_{k'_1})\cap R(\widetilde\phi_{k_2}) \cap R(\widetilde\phi'_{k'_2})$ joining $\widetilde \phi_{k_1}$ to  $\widetilde \phi_{k_2}$ and $\widetilde \phi'_{k'_1}$ to  $\widetilde \phi'_{k'_2}$ respectively, then $\gamma\cap\gamma'\not=\emptyset$.
 \end{itemize}

\begin{lemma} \label{lemma:unboundedness}  Every branch of $\widetilde f$ is unbounded (meaning non relatively compact) and meets infinitely many lifts of $\phi$. \end{lemma}

\begin{proof} Let $\widetilde\Gamma$ be a branch of $\widetilde f$ and $n$ a positive integer. Let $\widetilde z_*$ be a lift of  a point whose orbit is dense. Its $\phi$-trajectory being infinite contains an admissible sequence $(\widetilde \phi_k)_{1\leq k\leq n}$. As explained above, there exists $T\in G$ and $\widetilde z\in T(\widetilde\Gamma)$ such that $(\widetilde \phi_k)_{1\leq k\leq n}$ is a subsequence of the $\phi$-trajectory of $\widetilde z$. This implies that $T(\widetilde\Gamma)$ meets at least $n$ lifts of $\phi$. The branch $\widetilde\Gamma$ itself shares the same property. Consequently it meets infinitely many lifts of $\phi$ and must be unbounded.
\end{proof}

\begin{lemma} \label{lemma:distinct}  Two different branches of a fixed point $\widetilde z$ of $\widetilde f$ do not meet a common lift of $\phi$ except the lift of $\phi$ that contains $\widetilde z$ if such a lift exists. \end{lemma}

\begin{proof} A stable and an unstable branch of a fixed point $\widetilde z$ of $\widetilde f$ cannot meet both the same lift $\widetilde\phi$ of $\phi$ if $\widetilde z\not\in \widetilde \phi$. Indeed, one should have $\widetilde z\in R(\widetilde \phi)\cap  L(\widetilde \phi)$. Suppose now that the two unstable branches of  $\widetilde z$ meet the same lift $\widetilde\phi$ of $\phi$ and that $\widetilde z\not\in \widetilde \phi$. Write  $\widetilde x$ and $\widetilde y$ the first points where the two branches reach $\widetilde\phi$. One gets a simple closed curve, union of the segment of the unstable manifold of $\widetilde z$ that joins $\widetilde x$ to $\widetilde y$ and of the segment of $\widetilde \phi$ that joins $\widetilde y$ to $\widetilde x$. The bounded component of the complement of this curve must contain a stable branch of $\widetilde z$, in contradiction with Lemma \ref{lemma:unboundedness}. For the same reasons,  two stable branches of  $\widetilde z$ cannot meet the same lift $\widetilde\phi$ of $\phi$ if $\widetilde z\not\in \widetilde \phi$.\end{proof}

\begin{lemma} \label{lemma:choiceoflift} For every fixed point $\widetilde z$ of $\widetilde f$, every stable branch $\widetilde\Gamma^s$ of $\widetilde z$, every unstable branch $\widetilde\Gamma ^u$ of $\widetilde z$ and every finite admissible sequence $(\widetilde \phi_k)_{1\leq k\leq n}$, there exist $T$, $T'$ in $G$ such that
\begin{itemize}
\item $T(\widetilde z)\in R(\widetilde\phi_1)$;
\item $T'(\widetilde z)\in L(\widetilde\phi_n)$;
\item $(\widetilde \phi_k)_{1\leq k\leq n}$ is a subsequence of the $\phi$-trajectory of a point of $T(\widetilde \Gamma^u)$;
\item $(\widetilde \phi_k)_{1\leq k\leq n}$ is a subsequence of the $\phi$-trajectory of a point of $T'(\widetilde \Gamma^s)$.
\end{itemize}
\end{lemma}

\begin{proof}Let  $(\widetilde \phi_k)_{0\leq k\leq n+1}$ be an admissible extension of $(\widetilde \phi_k)_{1\leq k\leq n}$. There  exist $T$, $T'$ in $G$ such that
\begin{itemize}

\item $(\widetilde \phi_k)_{0\leq k\leq n+1}$ is a subsequence of the $\phi$-trajectory of a point $\widetilde x\in T(\widetilde \Gamma^u)$;
\item $(\widetilde \phi_k)_{0\leq k\leq n+1}$ is a subsequence of the $\phi$-trajectory of a point $\widetilde x'\in T'(\widetilde \Gamma^s)$.
\end{itemize}

If $n$ is sufficiently large, then $\widetilde f^{-n}(\widetilde x)\in R(\widetilde\phi_0)$ and $\widetilde f^{n}(\widetilde x')\in L(\widetilde\phi_{n+1})$ and so
$ T(\widetilde z)\in \overline{R(\widetilde\phi_0)}$ and $T'(\widetilde z)\in \overline{L(\widetilde\phi_{n+1})}$. To conclude, observe that $ \overline{R(\widetilde\phi_0)}\subset R(\widetilde\phi_1)$ and $ \overline{L(\widetilde\phi_{n+1})}\subset L(\widetilde\phi_{n})$.\end{proof}

\begin{lemma} \label{lemma:noncrossing} There is no crossing between admissible sequences.  \end{lemma}
\begin{proof} We argue by contradiction. Suppose that there exist two admissible paths $(\widetilde \phi_k)_{k\in J}$ and $(\widetilde \phi'_{k'})_{k'\in J'}$ and integers  $k_1<k_2$ in $J$ and $k'_1<k'_2\in J'$ such that:
 \begin{itemize}
 \item $\widetilde \phi_{k_1}\subset L(\widetilde\phi'_{k'_1})$ and $\widetilde \phi'_{k'_1}\subset L(\widetilde\phi_{k_1})$;
  \item $\widetilde \phi_{k_2}\subset R(\widetilde\phi'_{k'_2})$ and $\widetilde \phi'_{k'_2}\subset R(\widetilde\phi_{k_2})$;
    \item if $\gamma$ and $\gamma'$ are paths in the closure of $L(\widetilde\phi_{k_1})\cap L(\widetilde\phi'_{k'_1}) \cap R(\widetilde\phi_{k_2})\cap R(\widetilde\phi'_{k'_2})$ joining $\widetilde \phi_{k_1}$ to  $\widetilde \phi_{k_2}$ and $\widetilde \phi'_{k'_1}$ to  $\widetilde \phi'_{k'_2}$ respectively, then $\gamma\cap\gamma'\not=\emptyset$.
 \end{itemize}

Choose a fixed point $z$ of $f$,  denote $\Gamma_1^s$, $\Gamma_2^s$ its stable branches and $\Gamma_1^u$, $\Gamma_2^u$ its unstable branches. By Lemma  \ref{lemma:choiceoflift},  there exists a lift $\widetilde z\in R(\widetilde \phi_{k_1})$ of $z$ and a point of the branch $\widetilde \Gamma_1^u$ of $\widetilde z$ that lifts $\Gamma_1^u$, whose $\phi$-trajectory contains  $(\widetilde \phi_k)_{k_1\leq k\leq k_2}$.  Let $\widetilde y$ be the first point where $\widetilde\Gamma_1^u$ meets $\widetilde\phi_{k_2}$. Consider the other unstable branch $\widetilde\Gamma_2^u$ of $\widetilde z$ and denote $\widetilde \phi$ the first lift of $\phi$ met by $\widetilde\Gamma_2^u$  (different from the lift that contains $\widetilde z$ if such a lift exists). Denote $\widetilde x$ the first point where $\widetilde\Gamma_2^u$ meets $\widetilde\phi$.  The leaf $\widetilde\phi$ is contained in $R(\widetilde\phi_{k_1})$ by Lemma  \ref{lemma:distinct} and one gets a line $\lambda$ as the union of one of the half lines of $\widetilde\phi$ delimited by $\widetilde x$, the segment of the unstable manifold of $\widetilde z$ joining $\widetilde x$ to $\widetilde y$ and one of the half lines of $\widetilde\phi_{k_2}$ delimited by $\widetilde y$. Similarly,  there is  a lift $\widetilde z'$ of $z$ in $L(\widetilde\phi'_{k'_2})$ such that the stable branch $\widetilde\Gamma'{}^s_1$ of $\widetilde z'$ that lifts $\Gamma'{}^s_1$ intersects $R(\widetilde\phi'_{k'_1})$. Let $\widetilde y'$ be the first point where $\widetilde\Gamma'{}^s_1$ meets $\widetilde\phi'_{k'_1}$. Consider the other stable branch $\widetilde\Gamma'{}^s_2$ of $\widetilde z'$, the first lift $\widetilde\phi'$ of $\phi$ met by $\widetilde\Gamma'{}^s_2$  (different from the lift that contains $\widetilde z$ if such a lift exists) and the first point $\widetilde x'$ where $\widetilde\Gamma'{}^s_2$ meets $\widetilde\phi'$.  The leaf $\widetilde\phi'$ is contained in $L(\widetilde\phi'_{k'_2})$ and one gets a line $\lambda'$ as the union of one of the half lines of $\widetilde\phi'$ delimited by $\widetilde x'$, the segment of the stable manifold of $\widetilde z'$ joining $\widetilde x'$ to $\widetilde y'$ and one of the half lines of $\widetilde\phi'_{k'_1}$ delimited by $\widetilde y'$. The contradiction comes from the fact that $\lambda$ and $\lambda'$ must be disjoint. The key points in the previous proof  are the following:
 \begin{itemize}
 \item the unstable manifold of $\widetilde z$ does not meet the stable manifold of $\widetilde z'$;
\item the unstable manifold of $\widetilde z$ meets $\widetilde\phi_{k_1}$ and $\widetilde\phi_{k_2}$ but not $\widetilde\phi_{k'_1}$ because $\widetilde z \in L(\widetilde\phi_{k'_1})$;
\item  the unstable manifold of $\widetilde z$ could meet $\widetilde\phi_{k'_2}$ but not $\widetilde \phi'$ because  $\widetilde z \in L(\widetilde\phi')$;
\item the stable manifold of $\widetilde z'$ meets $\widetilde\phi_{k'_1}$ and $\widetilde\phi_{k'_2}$ but not $\widetilde\phi_{k_2}$ because $\widetilde z' \in R(\widetilde\phi_{k_2})$;
\item  the stable manifold of $\widetilde z'$ could meet $\widetilde\phi_{k_1}$ but not $\widetilde \phi$ because  $\widetilde z' \in R(\widetilde\phi)$.
\end{itemize}
\end{proof}

We deduce the following:

\begin{lemma} \label{lemma:comparable} The lifts of $\phi$ met by a stable or an unstable branch of a fixed point $\widetilde z$ of $\widetilde f$ (and different from the lift that contains $\widetilde z$ if such a lift exists) are all comparable (for the order $\preceq$). \end{lemma}

\begin{proof}We will give the proof for unstable branches, the case of stable branches being similar. Here again we argue by contradiction and suppose that  an unstable branch $\widetilde\Gamma^u$ of a fixed point $\widetilde  z$ meets two non comparable lifts $\widetilde\phi_2$ and $\widetilde\phi'_2$ of $\phi$, and that $\widetilde z\not \in \widetilde\phi_2\cup\widetilde\phi'_2$. Consequently $\widetilde z$ belongs to $R(\widetilde\phi_2)\cap R(\widetilde\phi'_2)$. Consider the two stable branches  $\widetilde \Gamma^s$ and $\widetilde \Gamma'^s$ of $\widetilde z$ and denote $\widetilde\phi_1$ and $\widetilde\phi'_1$  the first lifts of $\phi$ met by $\widetilde \Gamma^s\setminus\{\widetilde z\}$ and $\widetilde \Gamma'^s\setminus\{\widetilde z\}$ respectively. Note that the sets $\overline{R(\widetilde\phi_1)}$ and $\overline{ R(\widetilde\phi'_1)}$ are disjoint and both contained in $R(\widetilde\phi_2)$ and $R(\widetilde\phi'_2)$. By assumption one can find a point $\widetilde x\in \widetilde\Gamma^u\cap L(\widetilde \phi_2)$. Choose a neighborhood $\widetilde U\subset L(\widetilde \phi_2)$ of  $\widetilde x$. By the $\lambda$-lemma, one knows that there exists $n\geq 1$ such that $\widetilde f^{-n}(\widetilde U)\cap  R(\widetilde \phi_1) \not=\emptyset$ and $\widetilde f^{-n}(\widetilde U)\cap  R(\widetilde \phi'_1)\not=\emptyset$.  So there is a point in $R(\widetilde \phi_1) $  whose forward orbit reaches   $L(\widetilde \phi_2)$ and  a point in $R(\widetilde \phi'_1) $  whose forward orbit reaches   $L(\widetilde \phi_2)$. For similar reasons,   there is a point in $R(\widetilde \phi_1) $  whose forward orbit reaches   $L(\widetilde \phi'_2)$ and  a point in $R(\widetilde \phi'_1) $ whose forward orbit reaches   $L(\widetilde \phi'_2)$. By this implies that there is a crossing. \end{proof}

Consequently, to every unstable branch $\widetilde \Gamma^u$ of a fixed point $\widetilde z$ is associated an increasing sequence  $(\widetilde \phi_k)_{k\geq 1}$, where $\widetilde \phi_k$ is the $k$-th lift of $\phi$ met by $\widetilde \Gamma^u$, and distinct from the lift of $\phi$ that contains $\widetilde z$ if such a lift exists. Similarly, to every stable branch $\widetilde \Gamma^s$ is associated an increasing sequence  $(\widetilde \phi_k)_{k\leq -1}$. This sequence will be called the {\it $\phi$-trajectory} of the branch. It can be defined by the following:
\begin{itemize}
\item the $\phi$-trajectory of an orbit on the branch is a sub-sequence of the $\phi$-trajectory of the branch if it does not contain the eventual lift of $\widetilde\phi$ that contains $\widetilde z$;
\item every finite sequence of the $\phi$-trajectory of the branch is a sub-sequence of the $\phi$-trajectory of a least one orbit on the branch.
\end{itemize}
Note also, by Lemma \ref{lemma:choiceoflift}, that every finite admissible sequence is a sub-sequence of the $\phi$-trajectory of the image of the branch by a covering automorphism. To conclude with the remarks, note as well that the intersection of a branch with a lift is not necessarily closed (a branch has no reason to be proper).

\begin{lemma} \label{lemma:boundedness}  There exists an integer $A$ such that if $\widetilde\phi_0$ and $\widetilde \phi'_0$ are two different lifts of $\phi$, then the set of points of $\widetilde\phi_0$ that belong to a stable or an unstable branch of $\widetilde f$ that meets $\widetilde \phi'_0$ has length at most $A$. \end{lemma}

\begin{proof} Like in the proof of Proposition  \ref{prop:nonzerohomperiodic}, the line $\widetilde\phi_0$ being oriented, induces a natural order on the set of lifts of $\phi$ on the left of $\widetilde\phi_0$ and adjacent to $\widetilde\phi_0$ and  an order on the set of lifts of $\phi$ on the right of $\widetilde\phi_0$ and adjacent to $\widetilde\phi_0$ (we can say that a lift is above or below another lift relative to $\widetilde\phi_0$). There is no loss of generality by supposing that $\widetilde \phi'_0$ is on the left of $\widetilde\phi_0$ and adjacent to $\widetilde\phi_0$. The trivial sequence $(\widetilde\phi_k)_{k=0}$ being obviously admissible can be extended to an admissible sequence $(\widetilde\phi_k)_{-1\leq k\leq 1}$. Moreover, one can find an unstable branch $\widetilde\Gamma^u$ such that its $\phi$-trajectory contains $(\widetilde\phi_k)_{-1\leq k\leq 1}$. Let $\widetilde y$ be the first point where $\widetilde\Gamma^u$ meets $\widetilde\phi_{1}$ and $\widetilde x$ the last point where  $\widetilde\Gamma^u$ meets $\widetilde\phi_{-1}$ before reaching $\widetilde\phi_1$. We get a line $\widetilde\lambda$ by considering a half line of $\widetilde\phi_{-1}$ delimited by $\widetilde x$, then the segment $\widetilde \alpha$ of $\widetilde\Gamma^u$ joining $\widetilde x$ to $\widetilde y$, then a half line of $\widetilde\phi_{1}$ delimited by $\widetilde y$. The genus of $S$ being greater than $1$, the lifts $T_{\widetilde\phi_{0}}{}^k(\widetilde\phi_{-1})$, $k\in\Z$, are all distinct and contained in $R(\widetilde\phi_{0})$. Similarly, the lifts $T_{\widetilde\phi_{0}}{}^k(\widetilde\phi_{1})$, $k\in\Z$, are all distinct and contained in $L(\widetilde\phi_{0})$. Using Lemma  \ref{lemma:comparable}, one deduces that $\widetilde\lambda\cap T_{\widetilde \phi_0}{}^k(\widetilde\lambda)=\emptyset$ for every $k\not=0$. Denote $\widetilde z_-$ and $\widetilde z_+$ the smallest and the largest point of $\widetilde \alpha\cap \widetilde\phi_0$ for the order $\leq_{\widetilde\phi_0}$. There exists an integer $k$ such that  $\widetilde\phi'_0$ is above $T_{\widetilde\phi_0}{}^{k-1}(\widetilde\phi_1)$ relative to $\widetilde\phi_0$ and below $T_{\widetilde\phi_0}{}^{k+1}(\widetilde \phi_1)$ relative to $\widetilde\phi_0$.  By Lemma \ref{lemma:comparable}, a branch of $\widetilde f$ that meets $\widetilde\phi_0$ and $\widetilde\phi'_0$ cannot meet $T_{\widetilde\phi_0}{}^{k'}(\widetilde\phi_1)$ if $k'\not=k$. For the same reason it is different from $T_{\widetilde\phi_0}{}^{k'}(\widetilde\Gamma^u)$. Two different branches do not intersect, so a branch of $\widetilde f$ that meets $\widetilde\phi_0$ and $\widetilde\phi'_0$ is disjoint from $T_{\widetilde\phi_0}{}^{k'}(\widetilde\alpha)$  if $k'\not=k$. Finally, it meets at most one $T_{\widetilde\phi_0}{}^{k'}(\widetilde\phi_{-1})$, so it meets at most one line $T_{\widetilde\phi_0}{}^{k'}(\widetilde\lambda)$, $k'\not=k$. The lines $T_{\widetilde\phi_0}{}^{k-1}(\widetilde\lambda)$ and $T_{\widetilde\phi_0}{}^{k-2}(\widetilde\lambda)$ separate $\widetilde\phi'_0$ from every point of $\widetilde\phi_0$ smaller than $ T_{\widetilde\phi_0}{}^{k-2}(\widetilde z_-)$ and so every unstable branch that meets $\widetilde\phi_0$ and $\widetilde\phi'_0$  cannot intersect $\widetilde\phi_0$ at a point smaller than $ T_{\widetilde\phi_0}{}^{k-2}(\widetilde z_-)$. For the same reason, the lines $T_{\widetilde\phi_0}{}^{k+1}(\widetilde\lambda)$ and $T_{\widetilde\phi_0}{}^{k+2}(\widetilde\lambda)$ separate $\widetilde\phi'_0$ from every point of $\widetilde\phi_0$ larger than $ T_{\widetilde\phi_0}{}^{k+2}(\widetilde z_+)$ and so every branch that meets $\widetilde\phi_0$ and $\widetilde\phi'_0$  cannot intersect $\widetilde\phi_0$ at a point larger than $ T_{\widetilde\phi_0}{}^{k+2}(\widetilde z_+)$. We have proved that there exists $A$ such that the set of points of $\widetilde\phi_0$ that belong to a branch that meets $\widetilde \phi'_0$ has length $\leq A$.   \end{proof}

We are ready now to prove Proposition \ref{prop:nonhomologicalintersection}.
\begin{proof}[Proof of Proposition \ref{prop:nonhomologicalintersection}]
We will argue by contradiction and suppose that $$[\phi]\wedge \mathrm{rot}_f(\mu)>0, \enskip [\phi]\wedge \mathrm{rot}_f(\mu')> 0, \enskip \mathrm{rot}_f(\mu)\wedge \mathrm{rot}_f(\mu')\not= 0.$$
We consider a $\mu$-generic point $z$ and a $\mu'$-generic point $z'$, then two lifts $\widetilde z$, $\widetilde z'$ of $z$, $z'$ respectively. Perturbing $\phi$ if necessary,
we can always suppose that $z$ and $z'$ do not belong to $\phi$ \footnote{In fact, it is not necessary to perturb $\phi$, but notations are simpler to deal with in case the points do not belong to the leaf} . The $\phi$-trajectory of $\widetilde z$ and $\widetilde z'$ must be infinite because $[\phi]\wedge \mathrm{rot}_f(z)>0$ and $[\phi]\wedge \mathrm{rot}_f(z')>0$. We denote them $(\widetilde\phi_k)_{k\in\Z}$ and $(\widetilde\phi'_k)_{k\in\Z}$  respectively, supposing that $\widetilde z\in L(\widetilde \phi_0)\cap R(\widetilde \phi_1)$ and  $\widetilde z'\in L(\widetilde \phi'_0)\cap R(\widetilde \phi'_1)$. We can find a closed disk $\widetilde D\subset L(\widetilde \phi_0)\cap R(\widetilde \phi_1)$, neighborhood of $\widetilde z$ such $(\widetilde \phi_k)_{0\leq k\leq 1}$ is a subsequence of the $\phi$-trajectory of every point in $D$. We define similarly a neighborhood $D'$ of $\widetilde z'$. We can find $n$, $n'$ arbitrarily large and $S_n$, $S'_{n'}$  covering automorphisms such that

\begin{itemize}
\item
$\widetilde f^n(\widetilde z)\in S_n(\mathrm{int}(D))$,
\item $\widetilde f^{n'}(\widetilde z')\in S'_{n'}(\mathrm{int}(D'))$,
\item $[S_n]\sim n \,\mathrm{rot}_f(\mu)$,
\item $[S'_{n'}]\sim n' \,\mathrm{rot}_f(\mu')$,
\item $k_n=[\phi]\wedge[S_n]\sim n \,[\phi]\wedge \mathrm{rot}_f(\mu)$,
\item $k_{n'}=[\phi]\wedge[S'_{n'}]\sim n' \,[\phi]\wedge \mathrm{rot}_f(\mu')$,
\item $[S_n]\wedge [S'_{n'}]\sim n n' \,\mathrm{rot}_f(\mu)\wedge \mathrm{rot}_f(\mu')$.
\end{itemize}

Note that $\widetilde \phi_{k_n}=S_n(\widetilde \phi_{0})$ and $\widetilde \phi_{k_n+1}=S_n(\widetilde \phi_{1})$. As explained before, there exists an unstable branch $\widetilde\Gamma^u_n$ whose $\phi$-trajectory contains $(\widetilde \phi_k)_{0\leq k\leq k_{n+1}}$. Let $\widetilde y_n$ be the first intersection point of this branch with $\widetilde \phi_{k_n}$ and $\widetilde x_n$ the last point where the branch meets  $\widetilde \phi_{k_0}$ before reaching $\widetilde \phi_{k_n}$. Denote $\widetilde\alpha_n$ the segment of $\widetilde\Gamma^u_n$ that joins $\widetilde x_n$ to $\widetilde y_n$ and $\widetilde\beta_n$ the segment of $\widetilde \phi_{k_n}$ that joins $\widetilde y_n$ to $S_n(\widetilde x_n)$.  The branches   $\widetilde\Gamma^u_n$ and   $S_n(\widetilde\Gamma^u_n)$ meet $\widetilde \phi_{k_n}$ and $\widetilde \phi_{k_n+1}$ and so $\widetilde \beta_n$ has length $\leq A$ by Lemma  \ref{lemma:boundedness}. Note that $\widetilde\alpha_n$ and $\widetilde \beta_n$ projects onto paths $\alpha_n$ and $\beta_n$ and that $\Lambda_n =\alpha_n\beta_n$ is a loop that is lifted to a line $\widetilde\Lambda_n$, union of the translated of $\widetilde\alpha_n\widetilde \beta_n$ by the power of $S_n$.
There exists a stable branch $\widetilde\Gamma^s_{n'}$ whose $\phi$-trajectory contains $(\widetilde \phi'_{k'})_{0\leq k'\leq k'_{n+1}}$.  So we can define similarly a subpath $\widetilde\alpha'_{n'}$ of $\widetilde\Gamma^s_{n'}$ joining $\widetilde \phi'_{k'_0}$ to $\widetilde \phi'_{k'_{n'}}$, a sub-path $\widetilde\beta'_{n'}$ of  $\widetilde \phi'_{k'_{n'}}$, a line  $\widetilde \Lambda'_n$ and the projections $\alpha'_{n'}$, $\beta'_{n'}$, $\Lambda'_{n'}$. We will prove that
 $$\vert\Lambda_n\wedge \Lambda'_n\vert\leq 3A(k_n+k'_{n'})=O(n+n'),$$
 which contradicts the equality
  $$\Lambda_n\wedge \Lambda'_{n'} = [S_n]\wedge[S'_{n'}]\sim n n' \,\mathrm{rot}_f(\mu)\wedge \mathrm{rot}_f(\mu'),$$
 if  $n$ and $n'$ are large enough.
 The intersection number $\alpha_n\wedge \alpha'_{n'}$ is well defined and equal to zero, because a stable branch and an unstable branch do not intersect.  The intersection numbers $\alpha_n\wedge \beta'_{n'}$, $\beta_{n}\wedge \alpha'_{n'}$ and $\beta_n\wedge \beta'_{n'}$ are not necessarily defined but they are defined if we slightly enlarge $\beta_n$ and $\beta'_{n'}$ on $\Lambda_n$ and $\Lambda'_{n'}$ respectively, and slightly reduce $\alpha_n$ and $\alpha'_{n'}$. Let us do this, without changing the names of the paths.  It is sufficient to prove that
 $$\vert \alpha_n\wedge \beta'_{n'}\vert \leq 3A (k_n-1), \enskip \vert \beta_n\wedge \alpha'_{n'}\vert \leq 3A (k'_{n'}-1), \enskip \vert \beta_n\wedge \beta'_{n'}\vert \leq 3A.$$
We have the following formula, where the sum on the right has finitely many non zero terms
 $$\beta_n\wedge \beta'_{n'}=\sum_{S\in G}  \widetilde \beta_n\wedge S(\widetilde\beta'_{n'}).$$
Note now that each term in the sum belongs to $\{-1,0,1\}$. Note also that there are at most $3m$ non zero terms because $\widetilde \beta_n$ and $\widetilde\beta'_{n'}$ have length $\leq m$.
Similarly, one has
 $$\alpha_n\wedge \beta'_{n'}=\sum_{S\in G}  \widetilde \alpha_n\wedge S(\widetilde\beta'_{n'}).$$
Here again each term in the sum belongs to $\{-1,0,1\}$. Indeed, if $\widetilde \alpha_n\cap S(\widetilde\beta'_{n'})\not=\emptyset$, then $S(\widetilde\beta'_{n'})$ belongs to a lift $\widetilde \phi_{k}$, $1<k\leq k_n$,  and in that case $\widetilde \alpha_n\wedge S(\widetilde\beta'_{n'})=\widetilde \Lambda_n\wedge S(\widetilde\beta'_{n'})$. So it belongs to $\{-1,0,1\}$ because $\widetilde \Gamma_n$ is a line.  There are at most  $3A(k_n-1)$ non zero terms because the intersection of $\widetilde \alpha_n$ with a lift $\widetilde \phi_{k}$, $1<k\leq k_n$, has length $\leq A$ by Lemma  \ref{lemma:boundedness}, like $\widetilde\beta'_{n'}$.  The inequality $\vert \beta_n\wedge \alpha'_{n'}\vert \leq 3A (k'_{n'}-1)$ can be proven in the same way. \end{proof}

We will now generalize what has been done under a perturbative situation. Let $\mathcal H\subset \mathcal G^r_{\omega}(S)$ be a set satisfying the following:

\begin{itemize}

\item every $h\in \mathcal H$ coincide with $f$ in a neighborhood of the fixed point set;
\item for every $h\in \mathcal H$, the leaf $\phi$ is lifted into Brouwer lines (possibly singular)  of the natural lift of $h$;
\item  $\# \mathrm{per}_h(h)=2g-2$ for every $h\in{\mathcal H}$;
\item $\mathcal H$ is connected for the $C^0$-topology.
\end{itemize}

\begin{proposition} \label{prop:nonhomologicalintersectiongeneralized} Fix $h$ and $h'$ in $\mathcal H$ and $\mu$ and $\mu'$ ergodic probability measures invariant by $h$ and $h'$ respectively. If  $[\phi]\wedge \mathrm{rot}_h(\mu)>0$  and  $[\phi]\wedge \mathrm{rot}_{h'}(\mu')> 0$, then $\mathrm{rot}_h(\mu)\wedge\mathrm{rot}_{h'}(\mu')= 0$.   \end{proposition}

\begin{proof} To every branch $\widetilde\Gamma$ of a fixed point $\widetilde z$ of $\widetilde f$ and to every $h\in{\mathcal H}$, there exists a unique branch  $\widetilde\Gamma(h)$ of  $\widetilde z$ for $\widetilde h$ that coincides with $\widetilde\Gamma$ in a neighborhood of $\widetilde z$. We begin with this important result:

\begin{lemma} \label{lemma:persistence} If $\widetilde\Gamma$ is a branch  of a fixed point $\widetilde z$, the $\phi$-trajectory of $\widetilde\Gamma(h)$ for $\widetilde h$ does not depend on $h$.  \end{lemma}

\begin{proof} For every $h$, write $\widetilde \phi(h)$ the first lift of $\phi$ met by $\widetilde\Gamma(h)$ and different from the lift that contains $\widetilde z$ if such a lift exists. If $\widetilde \phi(h)\not=\widetilde \phi(h')$, then $\widetilde\Gamma(h)$ does not meet $\widetilde \phi(h')$ because $\widetilde \phi(h')$ and  $\widetilde \phi(h)$
are not comparable. We deduce that for every lift $\widetilde\phi$ of $\phi$, the set of $h\in \mathcal{H}$ such that $\widetilde\phi(h)=\widetilde\phi$ is open and closed in $ \mathcal{H}$ . Consequently, by connectedness of $ \mathcal{H}$, the first lift of $\phi$ met by $\widetilde\Gamma(h)$ is independent of $h$. The same argument permits to prove that for every $n\geq 1$, the $n$-th  lift of $\phi$ met by $\widetilde\Gamma(h)$ is independent of $h$.
\end{proof}

To prove Proposition \ref{prop:nonhomologicalintersectiongeneralized} , suppose that there exist $h$, $h'$ in $\mathcal H$ and $\mu$, $\mu'$ ergodic probability measures invariant by $h$, $h'$ respectively such that
$$[\phi]\wedge \mathrm{rot}_h(\mu)>0, \enskip [\phi]\wedge\mathrm{rot}_{h'}(\mu')> 0,\enskip \mathrm{rot}_h(\mu)\wedge \mathrm{rot}_{h'}(\mu')\not=0.$$   Like in the proof of Proposition \ref{prop:nonhomologicalintersection}, we can find $n$, $n'$ arbitrarily large and $S_n$, $S'_{n'}$  covering automorphisms such that
\begin{itemize}
\item $[S_n]\sim n \,\mathrm{rot}_h(\mu)$,
\item $[S'_{n'}]\sim n' \,\mathrm{rot}_{h'}(\mu')$,
\item $k_n=[\phi]\wedge[S_n]\sim n \,[\phi]\wedge\mathrm{rot}_h(\mu)$,
\item $k_{n'}=[\phi]\wedge[S'_{n'}]\sim n' \,[\phi]\wedge \mathrm{rot}_{h'}(\mu')$,
\item $[S_n]\wedge[S'_{n'}]\sim n n' \,\mathrm{rot}_h(\mu)\wedge \mathrm{rot}_{h'}(\mu')$,
\item an admissible sequence $(\widetilde\phi_k)_{0\leq k\leq k_{n+1}}$ of $\widetilde h$ such that $\widetilde\phi_{k_n}=S_n({\widetilde\phi_0})$ and $\widetilde\phi_{k_{n+1}}=S_n({\widetilde\phi_1})$,
\item an admissible sequence $(\widetilde\phi'_{k'})_{0\leq k'\leq k_{n'+1}}$ of $\widetilde h'$ such that $\widetilde\phi'_{k'_{n'}}=S'_{n'}(\widetilde\phi'_0)$ and $\widetilde\phi_{k'_{n'+1}}=S'_{n'}(\widetilde\phi'_1)$.
\end{itemize}
The sequence $(\widetilde\phi_k)_{0\leq k\leq k_{n+1}}$ is a sub-sequence of the $\phi$-trajectory of an unstable branch of $h$ and so by Lemma \ref{lemma:persistence} is a sub-sequence of the $\phi$-trajectory of an unstable branch of $\widetilde  f$. Similarly  $(\widetilde\phi'_k)_{0\leq k\leq k_{n'+1}}$ is a sub-sequence of the $\phi$-trajectory of a stable branch of $\widetilde f$. We are in the same situation as in the proof of Proposition \ref{prop:nonhomologicalintersection}. We have a contradiction.\end{proof}

\section{Proof of Theorem \ref{t.isotopy}.}\label{s:proof-isotopic-identity}

We fix $f\in  \mathcal{G}^r_{\omega}(S)$ and suppose that there exists $q\geq 1$ such that $f^q$ is isotopic to the identity. We want to prove that there exists $f'\in  \mathcal{G}^r_{\omega}(S)$, arbitrarily close to $f$, such that $\# \mathrm{per}_h(f')>\max(0, 2g-2)$.

If $g=1$, the result was already known and explained in the introduction. We will suppose from now on that $g\geq 2$.

As explained in the previous section, if  $\# \mathrm{per}_h(f)=2g-2$, there exists a simple loop $\phi$ that is lifted to Brouwer lines (possibly singular) of the natural lift $\widetilde f_q$ of $f^q$. The rotation number $\mathrm{rot}_{f^q}(\mu_{\omega})$ is not equal to zero and more precisely we have $[\phi]\wedge \mathrm{rot}_{f^q}(\mu_{\omega})>0$. In fact, the rotation number $\mathrm{rot}_{f^q}(z)$ is defined $\mu_{\omega}$-almost everywhere, and we have $[\phi]\wedge \mathrm{rot}_{f^q}(z)\geq 0$ with a strict inequality on a set of positive measure. Let $\widetilde  \phi$ be a lift of $\phi$,  there exists an open disk $U$ that admits a lift $\widetilde  U$ whose closure belongs to $R(\widetilde  \phi)\cap (\widetilde f_q)^{-1}(L(\widetilde \phi))$.

\begin{lemma} For $\mu_{\omega}$-almost every point in $V=\bigcup_{k\in\Z} f^{-kq}(U)$, we have $[\phi]\wedge \mathrm{rot}_{f^q}(z)> 0$.
\end{lemma}

\begin{proof}. The proof is classical. Write $\varphi:U\to U$ for the first return map of $f^{q}$ and $\tau :U\to \N\setminus\{0\}$ for  the time of first return. If $\mu$ is an invariant measure such that $\mu(U)>0$, the map $\varphi$ is defined $\mu$-almost everywhere on $U$ and preserves the measure $\mu_{\vert U}$. The map $\varphi$ is also defined $\mu$-almost everywhere and is $\mu$-integrable. Moreover, $\int_U \tau \, d\mu=\mu(V)$. One can construct a map $\rho : U\to H_1(M,\R)$ defined $\mu$-almost everywhere in the following way: if $\varphi(z)$ is well defined, one closes the trajectory $I^{\tau(z)-1}(z)$ with a path $\alpha$ contained in $U$ that joins $\varphi(z)$ to $z$, and set $\rho(z)=[I^{\tau(z)-1}(z)\alpha]$. The homology class of the loop $I^{\tau(z)-1}(z)\alpha$ is independent of the choice of $\alpha$. It is easy to prove that the map  $\rho/\tau$ is uniformly bounded on $U$ and consequently that $\rho$ is $\mu$-integrable. So, for $\mu$-almost every point, the Birkhoff means of $\rho$ and $\tau$ converges respectively to maps $\rho^*$ and $\tau^*$. If $\rho^*(z)$ and $\tau^*(z)$ are well defined, then the rotation vector $\mathrm{rot}(z)$ is well defined and equal to $\rho^*(z)/\tau^*(z)$. Note that the function $z\mapsto [\phi]\wedge\rho(z)$ is positive. One deduces that the fonction $z\mapsto [\phi]\wedge\rho^*(z)$ is also positive. The function $\tau^*$ being finite $\mu$-almost everywhere, one deduces that the map $z\mapsto [\phi]\wedge\mathrm{rot}(z)=  ([\phi]\wedge\rho^*(z))/\tau^*(z)$ is positive $\mu$-almost everywhere on $U$. Being invariant by $f^q$, it is positive on $V$. \end{proof}

\begin{proof}[Proof of Theorem \ref{t.isotopy}] By Proposition \ref{p.dynamicalproperties} , the complement of $V$ is included in a topological disk and so, by Alexander's trick, one can find a simple loop $\lambda\subset V$ homotopic to $\phi$. Consider a closed tubular neighborhood $W\subset V$  of $\lambda$. \footnote{In the case where $\phi$ is a non singular Brouwer line, Proposition \ref{p.dynamicalproperties} is not necessary, one can choose $\lambda$ to be equal to $\phi$  and $W$ to be a small neighborhood of $\phi$.}
 By compactness of $W$, there exists $K\in\Z$ such that $W\subset \bigcup_{-K\leq k\leq K} f^{-kq}(U)$. Moreover $W$ does not contain any fixed point of $f^q$. So, there exists a neighborhood $\mathcal W$ of $f$ in the set of homeomorphisms of $S$, furnished with the $C^0$-topology, such that if $f'$ belongs to $\mathcal W$ and coincides with $f$ outside $W$, then:

\begin{itemize}
\item  $W\subset \bigcup_{-K\leq k\leq K} f'{}^{-kq}(U)$;

\item $\phi$ is lifted to Brouwer lines (possibly singular) of the natural lift $\widetilde f'_q$ of $f'^q$.

\item  $\overline{ \widetilde U}\subset R(\widetilde \phi)\cap (\widetilde f'_q)^{-1}(L(\widetilde\phi))$.
\end{itemize}

One deduces that if $f'$ preserves $\mu_{\omega}$, then for $\mu_{\omega}$-almost every point in $W$, $ \mathrm{rot}_{f'{}^q}(z)$ is defined and satisfies $[\phi]\wedge \mathrm{rot}_{f'{}^q}(z)> 0$.

Consider now a divergence free smooth vector field supported on $W$ with an induced flow $(h_t)_{t\in\R}$ satisfying $\mathrm{rot}_{h_t}(\mu_{\omega})=t[\lambda]=t[\phi]$ and set $f_t=h_t\circ f$.  If $t$ is sufficiently small, then $f_t$ belongs to $\mathcal W$. So, for $\mu_{\omega}$-almost every point $z\in \bigcup_{k\in \Z} f_t{}^{-qk}(U)$, one has $[\phi]\wedge \mathrm{rot}_{f_t{}^q}(z)> 0$. But if $z\in V\setminus\bigcup_{k\in \Z} f'^{-qk}(U)$, then  $\mathrm{rot}_{f_t{}^q}(z)$ is defined if and only if $\mathrm{rot}_{f^q}(z)$ is defined and the two quantities are equal in that case. Consequently, for $\mu_{\omega}$-almost every point $z\in V$, one has $[\phi]\wedge \mathrm{rot}_{f_t{}^q}(z)> 0$.

By Corollary \ref{cor:periodic}, we know that if $t>0$ is small, then:
$$\begin{aligned}
\mathrm{rot}_{f_t{}^q} (\mu_{\omega}{}_{\vert V})\wedge \mathrm{rot}_{f^q} (\mu_{\omega}{}_{\vert V})&=q \,t[\phi]\wedge \mathrm{rot}_{f^q} (\mu_{\omega}{}_{\vert V})\\&=\int_V q\,t[\phi]\wedge \mathrm{rot}_{f^q} (z) \, d\mu_{\omega}(z)>0,
\end{aligned}$$and that
$$\mathrm{rot}_{f_t{}^q} (\mu_{\omega}{}_{\vert V})\wedge \mathrm{rot}_{f^q} (\mu_{\omega}{}_{\vert V})=\int_{V\times V}\mathrm{rot}_{f_t{}^q} (z)\wedge \mathrm{rot}_{f^q} (z') \, d\mu_{\omega}(z)d\mu_{\omega}(z').$$

Consequently, there exists an ergodic measure$\mu_t$ of $f_t$ and an ergodic measure  $\nu_t$ of $f$ such that
\begin{itemize}
\item $ [\phi]\wedge \mathrm{rot}_{f_t{}^q} (\mu_t)>0$;
\item $[\phi]\wedge \mathrm{rot}_{f^q} (\nu_t)>0$;
\item $\mathrm{rot}_{f_t{}^q} (\mu_t)\wedge \mathrm{rot}_{f^q} (\nu_t)>0$.
\end{itemize}

Fix $\varepsilon$ small. At least one of the following situations occurs:

\begin{itemize}
\item  the set $\mathcal{L}=\{f_t\,\vert\, t\in[0, \varepsilon]\}$ is included in $\mathcal{G}^r_{\omega}(S)$;

\item there exists  $t_0\in[0, \varepsilon]$ such that $\# \mathrm{per}(f_{t_0})>2g-2$;

\item there exists  $t_1\in[0, \varepsilon]$ such that $\# \mathrm{per}(f_{t_1})=2g-2$ and such that a stable and an unstable branch intersect.
\end{itemize}

In the first situation, we know by  Proposition \ref{prop:nonhomologicalintersectiongeneralized}  that there exists $t\in[0, \varepsilon]$ such that $\# \mathrm{per}_h(f_t)>2g-2$. In the second situation we can approximate $f_{t_0} $ by a map $f'\in \mathcal{G}^r_{\omega}(S)$ such that $\mathrm{per}_h(f')>2g-2$. In the last situation we can approximate $f_{t_1} $ by a map $f' \in\mathcal{G}^r_{\omega}(S)$ such that a stable and an unstable branch intersect and so we have $\mathrm{per}_h(f')>2g-2$ by Proposition \ref{p.dynamicalproperties}. In each situation we are done.
This concludes the proof of Theorem \ref{t.isotopy} and hence the proof of Theorem \ref{t.main}. \end{proof}

\section{An alternate proof of Proposition \ref{prop:nonhomologicalintersectiongeneralized} using forcing theory on transverse trajectories } \label{s:alternate}

Proposition \ref{prop:nonhomologicalintersection} is an immediate consequence of the much stronger following result of  G. Lellouch \cite{Lel}:

\begin{theorem}  \label{t:non-intersection} Let $S$ be a closed orientable surface of genus $g>1$. If a homeomorphism $h$ of $S$, isotopic to the identity, has two invariant measures  $\mu$ and $\mu'$ such that $\mathrm{rot}_h(\mu)\wedge \mathrm{rot}_h(\mu')\not= 0$, then it has a topological horseshoe. In particular $h$ has infinitely many periodic orbits and positive topological entropy.\end{theorem}

The proof of Theorem \ref{t:non-intersection} uses forcing theory on transverse trajectories of transverse foliations (see \cite{LecT}). In Section \ref{s:dynamics-isotopic-identity} we have been able to give a weaker version (Proposition  \ref{prop:nonhomologicalintersection}), sufficient for our purpose, by taking advantages of some properties of our map: the existence of a simple loop that is lifted to singular Brouwer lines,  the transitivity of the map, the denseness of the branches, the absence of homoclinic or heteroclinic intersections.
The key result for proving Theorem \ref{t.isotopy} is Proposition \ref{prop:nonhomologicalintersectiongeneralized} which is a generalization of Proposition \ref{prop:nonhomologicalintersection}. Similarly  (see \cite{Lel}) Theorem  \ref{t:non-intersection} can be generalized in the following way (explanations will be given later concerning the vocabulary):

\begin{theorem} \label{t:Lellouch}Let $S$ be a closed orientable surface of genus $g>1$ and $f$, $h$ two homeomorphisms of $S$ isotopic to the identity. We suppose that $I_f$ and $I_h$ are maximal isotopies of $f$ and $h$ respectively, satisfying the following:
\begin{itemize}
\item $I_f$ and $I_h$ have the same fixed point set;
\item there exists a foliation $\mathcal F$ that is transverse to $I_f$ and to $I_g$;
\item the admissible paths of $I_f$ and $I_h$ are the same;
\item there exists an invariant ergodic measure $\nu_f$ of $f$ and an invariant ergodic measure $\nu_h$ of $h$ such that $\rho_{f}(\nu_f)\wedge \rho_{h}(\nu_h)\not=0$.
\end{itemize}
 Then there exist two admissible paths (possibly equal) that intersect $\mathcal F$-transversally.
\end{theorem}

Previously to the proof of Theorem 1.8 given in Sections  \ref{s:dynamics-isotopic-identity}
 and \ref{s:proof-isotopic-identity}, we wrote a proof based on  forcing theory on transverse trajectories of transverse foliations and Theorem \ref{t:Lellouch}. We will briefly expose the proof, pointing the links with the arguments of Section \ref{s:dynamics-isotopic-identity}.

Fix $f\in\cG^r_{\omega}(S)$ isotopic to the identity, such that $\mathrm{per}_h(f)=2g-2$. By Lemma \ref{l.unlinkedfixedpoints}, one can find an isotopy $I_f$ from $\mathrm{Id}$ to $f$ that fixes every fixed point of $f$ (such an isotopy is uniquely defined up to homotopy). Let  $\mathcal F$ be a foliation transverse to $I_f$. It was recalled in Subsection \ref{ss.foliation} that for every point $z\not\in\mathrm{fix}(f)$, there exists a path $\gamma_z$ joining $z$ to $f(z)$, homotopic to $I_f(z)$ and positively transverse to $\mathcal F$, which means that every leaf of the foliation $\check {\mathcal F}$ obtained by lifting $\mathcal F\vert_{S\setminus \mathrm{fix}(f)}$ to the universal covering space $\check S$ of $S\setminus \mathrm{fix}(f)$ is a Brouwer line of the natural lift $\check f$ of $f\vert_{S\setminus \mathrm{fix}(f)}$.
The path $\gamma$ is not uniquely defined. Nevertheless, if $\gamma'$ is another choice, then $\gamma$ and $\gamma'$ can be lifted in $\check S$ to paths transverse to $\check{\mathcal{F}}$ that meet exactly the same leaves. We will say that $\gamma$ and $\gamma'$ are {\it equivalent}. We will write $\gamma=I_{\mathcal F}(z)$ and call this path the {\it transverse trajectory of $z$},  it is defined up to equivalence. For every integer $n\geq 1$ we can define
$I_{\mathcal F}^n(z)=\prod_{0\leq k<n} I_{\mathcal F}(f^k(z))$. Moreover, we can define
$$ I_{\mathcal F}^{+}(z)=\prod_{k\geq 0} I_{\mathcal F}(f^k(z)), \enskip I_{\mathcal F}^{-}(z)=\prod_{k<0} I_{\mathcal F}(f^k(z)), \enskip I_{\mathcal F}^{\pm}(z)=\prod_{k\in\Z} I_{\mathcal F}(f^k(z)).$$
We will say that a transverse path is {\it admissible} if it is equivalent to a path $I_{\mathcal F}^n(z)$, $n\geq 1$, $z\in S\setminus\mathrm{fix}(f)$.

The fact that $f$ is transitive implies that for every $z$ and $z'$, the paths $I_{\mathcal F}^{\pm}(z)$ and $I_{\mathcal F}^{\pm}(z')$ {\it do not intersect $\mathcal F$-transversally} (otherwise by the fundamental result of \cite{LecT} our map  would have positive topological entropy and infinitely many periodic points). This means that if $\check \gamma:\R\to \check S$ and $\check \gamma':\R\to \check S$ are lifts of $I_{\mathcal F}^{\pm}(z)$ and $I_{\mathcal F}^{\pm}(z')$ respectively, then there exist two transverse paths equivalent to $\check \gamma$ and $\check \gamma'$ respectively that do not intersect. In other words there is no crossing among the leaves met by $\check \gamma$ and $\check \gamma'$. Lemma \ref{lemma:noncrossing} is a reminder of this fact.

Every fixed point of $f$ being hyperbolic with fixed branches, a classification theorem of Le Roux \cite{Ler} tells us that the dynamics of $\mathcal F$ in a neighborhood of a fixed point $z_0$ is also saddle-like: it consists of four hyperbolic sectors separated by four parabolic sectors (possibly reduced to a single leaf) alternatively attracting and repelling. The link between the dynamics of $\mathcal F$ and the dynamics $f$ in a neighborhood of $z_0$ is expressed in the following result:
\begin{proposition}
\label{prop:branches}We have the following:

\begin{enumerate}

\item If  $\Gamma^s$ is a stable branch of $z_0$, there exists a neighborhood $\Gamma^s_{\mathrm {loc}} $ of $z_0$ in $\Gamma^s$ contained in a hyperbolic sector such that for every $z\in\Gamma^s_{\mathrm {loc}} $, the transverse  trajectory $I_{\mathcal F}^{+}(z)$ can be represented by a path joining  $z$ to $z_0$ and included in a hyperbolic sector.

\item If  $\Gamma^u$ is an unstable branch of $z_0$, there exists a neighborhood $\Gamma^u_{\mathrm {loc}} $ of $z_0$ in $\Gamma^u$ contained in a hyperbolic sector such that for every $z\in\Gamma^u_{\mathrm {loc}} $, the transverse  trajectory $I_{\mathcal F}^{-}(z)$ can be represented by a path  joining  $z_0$ to $z$ and included in a hyperbolic sector.

\item  every hyperbolic sector contains exactly one such a local branch.
\end{enumerate}
\end{proposition}

 Note that if $z$ and $z'$ belong to $\Gamma^s_{\mathrm {loc}} $, then of one the paths $I_{\mathcal F}^{+}(z)$, $I_{\mathcal F}^{+}(z')$ is a subpath of the other one (up to equivalence). Nevertheless, saying that $I_{\mathcal F}^{+}(z)$ is a subpath of $I_{\mathcal F}^{+}(z')$ does not mean that $z$ is closer to $z_0$ than $z'$ on $\Gamma^s_{\mathrm {loc}} $. Indeed $\Gamma^s_{\mathrm {loc}} $ is not necessarily transverse to the foliation.  The following result is much stronger:

\begin{lemma}
\label{l:transversebranch}
For every stable leaf $\Gamma^s$ and unstable leaf $\Gamma^u$ of $z_0$,  there exists transverse paths  $\Gamma^s_{\mathcal F}:\R\to S\setminus\mathrm{fix}(f)$ and $\Gamma^u_{\mathcal F}:\R\to S\setminus\mathrm{fix}(f) $, uniquely defined (up to reparametrization and equivalence) such that: \begin{itemize}
\item every trajectory $I_{\mathcal F}^{\pm}(z)$, $z\in \Gamma^s$ is a subpath of $\Gamma^s_{\mathcal F}$;
\item every trajectory $I_{\mathcal F}^{\pm}(z)$, $z\in \Gamma^u$ is a subpath of $\Gamma^u_{\mathcal F}$;
\item for every $m\leq 0$, there exists $z\in\Gamma^s$ such that $\Gamma^s_{\mathcal F}\vert_{[-m, +\infty)}$ is a subpath of $I_{\mathcal F}^{\pm}(z)$;
\item for every $m\geq 0$, there exists $z\in\Gamma^u$ such that $\Gamma^u_{\mathcal F}\vert_{(-\infty, m]}$ is a subpath of $I_{\mathcal F}^{\pm}(z)$.
\end{itemize}
\end{lemma}
To prove this lemma, it is sufficient to prove that if two points $z$ and $z'$ are on the same stable or unstable branch then either $I_{\mathcal F}^{\pm}(z)$ is  a subpath of $I_{\mathcal F}^{\pm}(z')$ or $I_{\mathcal F}^{\pm}(z')$ is a subpath of $I_{\mathcal F}^{\pm}(z)$. If it is not the case (for an unstable branch), by an argument very similar to what is done in the proof of Lemma  \ref{lemma:comparable} , one can find two points $w$ and $w'$ close to $z$ and $z'$ respectively such that $I_{\mathcal F}^{+}(w)$ and $I_{\mathcal F}^{+}(w')$ have a $\mathcal F$-transverse intersection, which is impossible.

We will call $\Gamma^s_{\mathcal F}$ the {\it transverse stable branch} associated to $\Gamma^s$ and $\Gamma^u_{\mathcal F}$ the {\it transverse unstable branch} associated to $\Gamma^u$.  Using the denseness of the branches, we get the following properties of the transverse branches (note that similar results have been proven in Section \ref{s:dynamics-isotopic-identity}):

\begin{lemma} \label{l:transverse-denseness}
Let $\Gamma_{\mathcal F}$ be a transverse branch. Then

\begin{enumerate}
\item every admissible path is equivalent to a subpath of  $\Gamma_{\mathcal F}$;

\item  $\Gamma_{\mathcal F}$ crosses every leaf of $\mathcal F$ infinitely many often.
\end{enumerate}
\end{lemma}

We will now generalize what has been done under a perturbative situation. Let $\mathcal H\subset \mathcal G_r$ be a set satisfying the following:

\begin{itemize}

\item every $h\in \mathcal H$ coincide with $f$ in a neighborhood of the fixed point set;
\item  $\# \mathrm{per}(h)=\# \mathrm{fix}_h(h)=2g-2$ for every $h\in{\mathcal H}$;
\item for every $h\in \mathcal H$, the foliation $\mathcal F$ is transverse to $h$ ;
\item $\mathcal H$ is connected for the $C^0$-topology.
\end{itemize}

Let $\Gamma$ be a branch of $f$ and $z$ the fixed point associated to this branch. For every $h\in\mathcal H$, there is a branch $\Gamma(h)$ that coincide with $\Gamma$ in a neighborhood of $z$. In other words, we have $\Gamma(h)_{\mathrm{loc}}= \Gamma_{ \mathrm{loc}}$. Moreover we can define the associated transverse branch $\Gamma_{\mathcal F}(h)$. We have the following result analogous to Lemma \ref{lemma:persistence}:

\begin{lemma} \label{l:transverse-branches} For every $h\in \mathcal{H}$, the branch $\Gamma_{\mathcal F}(h)$ is equivalent to $\Gamma_{\mathcal F}$ and consequently, the admissible transverse paths of $f$ and $h$ are the same.
\end{lemma}

It remains to use Theorem \ref{t:Lellouch} to get Theorem \ref{t.isotopy}.

\end{document}